\documentclass[12pt,amstex]{amsart}
\usepackage{stmaryrd}

\usepackage{epsfig}
\usepackage{amsmath}
\usepackage{amssymb}
\usepackage{amscd}
\usepackage{graphicx}

\usepackage{pstricks}

\usepackage{tocvsec2}

\usepackage{hyperref}

\input xy
\xyoption{all}

\newtheorem{thm}{Theorem}[section]
\newtheorem{lem}[thm]{Lemma}

\newtheorem{prop}[thm]{Proposition}

\newtheorem{cor}[thm]{Corollary}
\theoremstyle{definition}
\newtheorem{de}[thm]{Definition}

\theoremstyle{remark}
\newtheorem{rem}[thm]{Remark}

\numberwithin{equation}{section}

\def \N {\mathbb N}
\def \C {\mathbb C}
\def \Z {\mathbb Z}
\def \R {\mathbb R}
\def \A {\mathbb{A}}
\def \E {\mathbb{E}}

\def \B {\mathcal{B}}
\def \F {\mathcal{F}}
\def \G {\mathcal{G}}
\def \I {\mathcal{I}}
\def \J {\mathcal{J}}

\def \X {\mathcal{X}}
\def \Y {\mathcal{Y}}
\def \ZZ {\mathcal{Z}}
\def \O {\mathcal{O}}

\def \Q {{\bf Q}}

\def \id {{\rm id}}
\def \h {\hat }

\def \HK {\interleave}

\def \a {\alpha }
\def \b {\beta}
\def \ep {\epsilon}
\def \d {\delta}
\def \D {\Delta}
\def \ll {\lambda}

\def \lra{\longrightarrow}

\begin{document}
\title[Strictly ergodic models and convergence]{Strictly ergodic models and
the convergence of non-conventional pointwise ergodic averages}

\author{Wen Huang}
\author{Song Shao}
\author{Xiangdong Ye}

\address{Department of Mathematics, University of Science and Technology of China,
Hefei, Anhui, 230026, P.R. China.}

\email{wenh@mail.ustc.edu.cn}\email{songshao@ustc.edu.cn}
\email{yexd@ustc.edu.cn}

\subjclass[2000]{Primary: } 

\thanks{Huang is partially supported by NNSF for Distinghuished Young School (11225105),
Shao and Ye is partially supported by NNSF of China (11171320), and Huang and Ye is partially supported by NNSF of China (11371339).}

\date{June 09, 2010}
\date{June 27, 2010}
\date{July 17, 2010}
\date{August 01, 2010}
\date{October 09, 2010}
\date{April 30, 2012}
\date{May 03, 2012}
\date{July, 30, 2013}
\date{Oct. 22, 2013}
\date{Nov. 12, 2013}

\begin{abstract}

The well-known Jewett-Krieger's Theorem states that each ergodic
system has a strictly ergodic model. Strengthening the model by requiring that it
is strictly ergodic under some group actions, and
building the connection of the new model with the convergence of
pointwise non-conventional ergodic averages we prove that for an
ergodic system $(X,\X,\mu, T)$, $d\in\N$, $f_1, \ldots, f_d \in
L^{\infty}(\mu)$, 
the averages
\begin{equation*}
    \frac{1}{N^2} \sum_{(n,m)\in F_N}
    f_1(T^nx)f_2(T^{n+m}x)\ldots f_d(T^{n+(d-1)m}x)
\end{equation*}
converge $\mu$ a.e. We remark that the same method can be used to
show the pointwise convergence of ergodic averages along cubes which
was firstly proved by Assani and then extended to a general case by Chu and Franzikinakis.

\end{abstract}

\maketitle





\section{Introduction}\label{sec-1}
In the introduction we will state the main results of the paper and
give some backgrounds.

\subsection{Main results} Throughout this paper, by a {\it topological dynamical
system} (t.d.s. for short) we mean a pair $(X, T)$, where $X$ is a
compact metric space  and $T$ is a homeomorphism from $X$ to itself.
A {\em measurable system} (m.p.t. for short) is a quadruple $(X,\X,
\mu, T)$, where $(X,\X,\mu )$ is a Lebesgue probability space and $T
: X \rightarrow X$ is an invertible measure preserving
transformation.

Let $(X,\X, \mu, T)$ be an ergodic m.p.t. We say that  $(\h{X}, T)$
is a {\em topological model} (or just a {\em model}) for $(X,\X,
\mu, T)$ if $(\h{X}, T)$ is a t.d.s. and there exists an invariant
probability measure $\h{\mu}$ on the Borel $\sigma$-algebra
$\mathcal{B}(\h{X})$ such that the systems $(X,\X, \mu, T)$ and
$(\h{X}, \mathcal{B}(\h{X}), \h{\mu}, T)$ are measure theoretically
isomorphic.

\medskip

The well-known Jewett-Krieger's theorem  \cite{Jewett, Krieger}
states that every ergodic system has a strictly ergodic model. We
note that one can add some additional properties to the topological
model. For example, in \cite{Lehrer} Lehrer showed that the strictly
ergodic model can be required to be a topological (strongly) mixing
system in addition.

Let $(\h{X},T)$ be a t.d.s. Write $(x,\ldots,x)$ ($2^d$ times) as
$x^{[d]}$. Let $\F^{[d]}, \G^{[d]}$ and $\Q^{[d]}(\h{X})$ be the
face group of dimension $d$, the parallelepiped group of dimension
$d$ and the dynamical parallelepiped of dimension $d$ respectively
(see Section \ref{section-topo} for definitions). The orbit closure of $x^{[d]}$ under the face group
action will be denote by $\overline{\F^{[d]}}(x^{[d]})$. It was
shown by Shao and Ye \cite{SY} that if $(\h{X},T)$ is minimal then
$(\overline{\F^{[d]}}(x^{[d]}),\F^{[d]})$ is minimal for all $x\in
\hat{X}$ and $(\Q^{[d]}(\h{X}), \G^{[d]})$ is minimal.

In this paper we will strengthen Jewett-Krieger's theorem in another
direction. Namely, we have the following Theorem A and Theorem B.

\medskip
\noindent {\bf Theorem A:} {\em Let $(X,\X, \mu, T)$ be an ergodic
m.p.t. and $d\in\N$. Then
\begin{enumerate}
\item it has a strictly ergodic model $(\h{X}, T)$ such that $(\overline{\F^{[d]}}(x^{[d]}),\F^{[d]})$ is strictly ergodic for
all $x\in \hat{X}$.

\item it has a strictly ergodic model $(\h{X}, T)$ such that
$(\Q^{[d]}(\h{X}), \G^{[d]})$ is strictly ergodic.

\end{enumerate}}

\medskip

Now let $\tau_d=T\times \ldots\times T \ (d \ \text{times})$ and
$\sigma_d =T\times \ldots \times T^{d}$. The group generated by
$\tau_d$ and $\sigma_d$ is denoted $\langle\tau_d, \sigma_d\rangle$.
For any $x\in \h{X}$, let $N_d(\h{X},x)=\overline{\O((x,\ldots,x),
\langle\tau_d, \sigma_d\rangle)}$, the orbit closure of
$(x,\ldots,x)$ ($d$ times) under the action of the group
$\langle\tau_d, \sigma_d\rangle$. We remark that if $(\h{X},T)$ is
minimal, then all $N_d(\h{X},x)$ coincide, which will be denoted by
$N_d(\h{X})$. It was shown by Glasner \cite{G94} that if $(\h{X},T)$
is minimal, then $(N_d(\h{X}), \langle\tau_d, \sigma_d\rangle)$ is
minimal.

\medskip

\noindent {\bf Theorem B:} {\em Let $(X,\X, \mu, T)$ be an ergodic
m.p.t. and $d\in\N$. Then it has a strictly ergodic model $(\h{X}, T)$ such that
$(N_d(\h{X}), \langle\tau_d, \sigma_d\rangle)$ is strictly ergodic.
}

\medskip
We note that we have formulas to compute the unique measure in Theorems A and~ B.
Particularly, when $(X,\X, \mu, T)$ is weakly mixing, the unique measure is nothing but
the product measure. Moreover, for small $d$ we also have explicit description of the unique measure.

\medskip
Surprisedly, Theorems A and B are closely related the pointwise
convergence of non-conventional multiple ergodic averages. That is,
we can show Theorems C and D as applications of Theorems A and B
respectively.

\medskip
\noindent {\bf Theorem C:} {\em Let $(X,\X, \mu, T)$ be an ergodic
m.p.t. and $d\in \N$. Then

\begin{enumerate}
\item for functions $f_\ep\in L^{\infty}(\mu), \
\ep\in \{0,1\}^d, \ep\not=(0,\ldots,0) $, the averages
\begin{equation}\label{C1}
    \frac{1}{N^d} \sum_{{\bf n}\in \{0,1, \ldots,N-1\}^d} \prod _{(0,\ldots,0)\neq\ep\in \{0,1\}^d}
    f_\ep (T^{{\bf n}\cdot \ep}x)
\end{equation}
converge $\mu$ a.e..

\item  for functions $f_\ep\in L^{\infty}(\mu), \ \ep\in \{0,1\}^d$, the averages
\begin{equation}\label{C2}
    \frac{1}{N^{d+1}} \sum_{{\bf n}\in \{0,1,\ldots,N-1\}^d \atop{n\in \{0,1,\ldots,N-1\}}}
    \prod _{\ep\subset [d]}
    f_\ep (T^{n+{\bf n}\cdot \ep}x)
\end{equation}
converge $\mu$ a.e..
\end{enumerate}}

\medskip
\noindent {\bf Theorem D:} {\em Let $(X,\X,\mu, T)$ be an ergodic
m.p.t. and $d\in\N$. Then for $f_1, \ldots, f_d \in L^{\infty}(\mu)$
the averages
\begin{equation}\label{D}
    \frac{1}{N^2} \sum_{n,m\in [0,N-1]}
    f_1(T^nx)f_2(T^{n+m}x)\ldots f_d(T^{n+(d-1)m}x)
\end{equation}
converge $\mu$ a.e.}

As we said above we have formulas to compute the limits.
For example the limit is Theorem D is
$\int_{N_d({X})} \bigotimes f_i\text{d}\lambda_{\tau,\sigma;d}$, where $\lambda_{\tau,\sigma;d}$ is defined in (\ref{limit-D})
and we assume $(X,\X,\mu, T)$ itself is the model defined in Theorem B.

\subsection{Backgrounds}
In this subsection we will give backgrounds of our research.

\subsubsection{Topological model}
The pioneering work on topological model was done by Jewett in
\cite{Jewett}. He proved the theorem under the additional assumption
that $T$ is weakly mixing and
conjectured that if the condition of being weakly mixing is
replaced by that of being ergodic, the theorem would still be valid.
Jewett's conjecture was proved by Krieger in \cite{Krieger} soon.
This was followed by the papers of Hansel and Raoult \cite{HanRao}
and Denker \cite{Denker}, giving different proofs of the theorem in
the general ergodic case (see also \cite{DCS}).
Bellow and Furstenberg \cite{BeF} showed how with an additional
piece of information the Key Lemma in Jewett's paper -- and hence
Jewett's whole proof -- carries over to the general ergodic case.
One can add some additional properties to the topological model. For example, in \cite{Lehrer}
Lehrer showed that the strictly ergodic model can be required as a
topological (strongly) mixing system in addition. Our Theorems A and B strengthen Jewett-Krieger Theorem
in other direction, i.e. we can require the model to be well behavioral under some group actions.

It is well known that each m.p.t. has a topological model
\cite{F}. There are universal models, models for some group actions and models for some special classes.
Weiss \cite{Weiss89} showed the following nice result:
There exists a minimal t.d.s. $(X, T)$ with the property that for
every aperiodic ergodic m.p.t. $(Y,\Y,\nu, S)$ there exists a
$T$-invariant Borel probability measure $\mu$ on $X$ such that the
systems $(Y,\Y,\nu, S)$ and $(X,\B(X),\mu, T)$ are measure
theoretically isomorphic. Note that there exists universal model for all ergodic m.p.t.
with entropy less than or equal to a given  number $t>0$ \cite{SW} and it is interesting that there is no such a
model for zero entropy m.p.t. \cite{SJ}. Weiss \cite{Weiss85} showed that Jewett-Krieger Theorem can be
generalized from $\Z$-actions to commutative group actions (in
\cite{Weiss85} there is only an outline of a proof, and the
exposition of his proof can be found in \cite{Weiss00}, more details
can be found in \cite{Glasner, GW06}). An ergodic system has a doubly
minimal model if and only if it has zero entropy \cite{Weiss95}
(other topological models for zero entropy systems can be found in
\cite{HM, DL}); and an ergodic system has a strictly ergodic, UPE
(uniform positive entropy) model if and only if it has positive
entropy \cite{GW94}.

\medskip

We say that $\h{\pi}: \h{X}\rightarrow \h{Y}$ is a {\em topological
model} for $\pi: (X,\X, \mu, T)\rightarrow (Y,\Y, \nu, T)$ if
$\h{\pi}$ is a topological factor map and there exist measure
theoretical isomorphisms $\phi$ and $\psi$ such that the diagram
\[
\begin{CD}
X @>{\phi}>> \h{X}\\
@V{\pi}VV      @VV{\h{\pi}}V\\
Y @>{\psi }>> \h{Y}
\end{CD}
\]
is commutative, i.e. $\h{\pi}\phi=\psi\pi$. Weiss \cite{Weiss85}
generalized the theorem of Jewett-Krieger to the relative case.
Namely he proved that if $\pi: (X,\X, \mu, T)\rightarrow (Y,\Y, \nu,
T)$ is a factor map with $(X,\X, \mu, T)$ ergodic and
$(\h{Y},\h{\Y}, \h{\nu}, T)$ is a uniquely ergodic model for $(Y,\Y,
\nu, T)$, then there is a uniquely ergodic model $(\h{X}, \h{\X},
\h{\mu}, T)$ for $(X,\X, \mu, T)$ and a factor map $\h{\pi}:
\h{X}\rightarrow \h{Y}$ which is a model for $\pi: X\rightarrow Y$.
We will refer this theorem as {\it Weiss's Theorem}.
We note that in \cite{Weiss85} Weiss pointed that the relative case holds
for commutative group actions.


\subsubsection{Ergodic averages} In this subsection we recall some
results related to pointwise ergodic averages.

The first pointwise ergodic theorem was proved by Birkhoff in 1931.
Followed from Furstenberg's work in 1977, problems concerning the convergence of
multiple ergodic averages (in $L^2$ or pointwisely) become a very important part of the study
of ergodic theory.
\medskip

The convergence of the averages
\begin{equation}\label{multiple1}
    \frac 1 N\sum_{n=0}^{N-1}f_1(T^nx)\ldots
f_d(T^{dn}x)
\end{equation}
in $L^2$ norm was established by Host and Kra \cite{HK05} (see also Ziegler \cite{Z}). We note that in their
proofs, the characteristic factors play a great role. The multiple ergodic average for commuting transformations was obtained by Tao
\cite{Tao} using finitary ergodic method, see \cite{Austin,H} for
more traditional ergodic proofs. Recently, convergence of multiple ergodic averages for nilpotent group actions
was obtained by Walsh \cite{Walsh}.


\medskip

The first breakthrough on pointwise convergence of (\ref{multiple1}) for $d > 1$ is due to Bourgain, who showed in \cite{B90} that for $d = 2$,
the limit in (\ref{multiple1}) exists a.e. for all $f_1, f_2 \in  L^\infty$. It is a big open question if the same holds
for $d>2$. Very recently, Assani claimed the convergence for weakly mixing transformations \cite{Assani13}.

The study of the limiting behavior of the averages along cubes was
initiated by Bergelson in \cite{Bergelson00}, where convergence in
$L^2(\mu)$ was shown in dimension 2. Bergelson's result was later
extended by Host and Kra for cubic averages of an arbitrary
dimension $d$ in \cite{HK05}. More recently in \cite{Assani}, Assani
established pointwise convergence for cubic averages of an arbitrary
dimension $d$. Chu and Franzikinakis
\cite{CF} extended the result to a very general case, i.e. they showed that for measure preserving
transformations $T_\ep: X \rightarrow X$, functions $f_\ep\in
L^{\infty}(\mu), \ (0,\ldots,0) \neq \ep\in \{0,1\}^d$, the averages
\begin{equation*}
    \frac{1}{N^d} \sum_{{\bf n}\in [0,N-1]^d} \prod _{(0,\ldots,0) \neq \ep\in \{0,1\}^d}
    f_\ep (T_\ep^{{\bf n}\cdot \ep}x)
\end{equation*}
converge $\mu$ a.e.. Moreover, they obtained in the same paper that
\begin{equation*}
    \frac{1}{Nb(N)} \sum_{1\le m\le N,1\le n\le b(N)}
    f_1(T^{m+n}x)f_2(T^{m+2n}x)\ldots f_d(T^{m+dn}x)
\end{equation*}
converges pointwisely, where $b(N)/N^{1/d}\lra 0$ as $N \lra\infty$.

We remark that our method to prove Theorem D does not apply the general case
as shown by Chu and Franzikinakis in \cite{CF}. The advantage of our method
is that we can give formulas for the limits, meanwhile this can not obtained in
\cite{Assani, CF}.

\subsection{Main ideas of the proofs}

Now we describe the main ideas and ingredients in the proof of Theorem A (the proof of Theorem B will follow by the similar idea).
The first fact we face is that for an ergodic m.p.t. $(X,\X, \mu, T)$,
not every strictly ergodic model is its $\F^{[d]}$-strictly
ergodic model. For example, let $(X,\X, \mu, T)$ be a Kronecker
system. By Jewett-Krieger' Theorem, we may assume that $(X,T)$
is a topologically weakly mixing minimal system and strictly
ergodic. By \cite[Theorem 3.11]{SY}
$(\overline{\F^{[d]}}(x^{[d]}),\F^{[d]})$ is minimal for all $x\in
X$ and $\overline{\F^{[d]}}(x^{[d]})=\{x\}\times X^{[d]}_*$. It is
easy to see that $\d_x\times \mu^{\bigotimes 2^d-1}$ and
$\mu^{[d]}_*$ are two different invariant measures on it (see Section 2 for the definitions). This indicates that
to obtain Theorem A, Jewett-Krieger' Theorem is not enough for our purpose. Fortunately, we find that
Weiss's Theorem \cite{Weiss85} is a right tool.

Precisely, let $\pi_d: X\rightarrow Z_d$ be the factor map from
$X$ to its $d$-step nilfactor $Z_d$. By definition, $Z_d$ may be
regarded as a topological system in the natural way. By Weiss's Theorem
there is a uniquely ergodic model $(\h{X}, \h{\X},
\h{\mu}, T)$ for $(X,\X, \mu, T)$ and a factor map $\h{\pi_d}:
\h{X}\rightarrow Z_d$ which is a model for $\pi_d: X\rightarrow
Z_d$.
\[
\begin{CD}
X @>{\phi}>> \h{X}\\
@V{\pi_d}VV      @VV{\h{\pi_d}}V\\
Z_d @>{ }>> Z_d
\end{CD}
\]

We then show (though it is difficult) that $(\h{X},T)$ is what we need. To do this we heavily use the
theory of joinings (for a reference, see \cite{Glasner})
and some facts related to $d$-step nilsystems. Once Theorem A (resp. B) is proven, Theorem C (resp. D) will follow by an argument using some
well known theorems related to pointwise convergence for $\Z^d$ actions by and for uniquely ergodic systems.

We remark that currently we do not know how to prove the popintwise convergence of (\ref{multiple1})
using  similar ideas.

\subsection{Organization of the paper}

In Section \ref{section-topo}, we give basic notions and facts about
dynamical parallelepipeds and characteristic factors. In Section
\ref{sec-model} we define $\F$ and $\G$-strictly
ergodic models and prove that each ergodic system has $\F$ and
$\G$-strictly ergodic model. Moreover, we build the connection between $\F$ and
$\G$-strictly ergodic models with pointwise  convergence of averages
along cubes and faces, and deduce the existence of the limit of the averages. In the two sections followed, we study arithmetic
progression models and prove pointwise ergodic theorem along arithmetic progressions.

\section{Dynamical parallelepipeds and characteristic factors}\label{section-topo}

In this section we introduce basic knowledge about dynamical
parallelepipeds and characteristic factors. For more details, see
\cite{HK05,HK09, HKM} etc.

\subsection{Ergodic theory and topological dynamics}

In this subsection we introduce some basic notions in ergodic theory
and topological dynamics. For more information, see Appendix.

\subsubsection{Measurable systems}



For a m.p.t. $(X,\X, \mu, T)$ we write $\I=\I (T)$ for the $\sigma$-algebra $\{A\in \X : T^{-1}A =
A\}$ of invariant sets. A m.p.t. is {\em ergodic} if all
the $T$-invariant sets have measure either $0$ or $1$. $(X,\X, \mu,
T)$ is {\em weakly mixing} if the product system $(X\times X,
\X\times \X, \mu\times \mu, T\times T)$ is erdogic.

\medskip

A {\em homomorphism} from m.p.t. $(X,\X, \mu, T)$ to $(Y,\Y, \nu,
S)$ is a measurable map $\pi : X_0 \rightarrow  Y_0$, where $X_0$ is
a $T$-invariant subset of $X$ and $Y_0$ is an $S$-invariant subset
of $Y$, both of full measure, such that $\pi_*\mu=\mu\circ
\pi^{-1}=\nu$ and $S\circ \pi(x)=\pi\circ T(x)$ for $x\in X_0$. When
we have such a homomorphism we say that $(Y,\Y, \nu, S)$
is a {\em factor} of  $(X,\X, \mu , T)$. If the factor map
$\pi: X_0\rightarrow  Y_0$ can be chosen to be bijective, then we
say that $(X,\X, \mu, T)$ and $(Y,\Y, \nu, S)$ are {\em
(measure theoretically) isomorphic} (bijective maps on Lebesgue
spaces have measurable inverses). A factor can be characterized
(modulo isomorphism) by $\pi^{-1}(\Y)$, which is a $T$-invariant
sub- $\sigma$-algebra of $\X$, and conversely any $T$-invariant
sub-$\sigma$-algebra of $\X$ defines a factor. By a classical result abuse
of terminology we denote by the same letter the $\sigma$-algebra
$\Y$ and its inverse image by $\pi$. In other words, if $(Y,\Y, \nu,
S)$ is a factor of $(X,\X, \mu, T)$, we think of $\Y$ as a
sub-$\sigma$-algebra of $\X$.

\subsubsection{Topological dynamical systems}


A t.d.s. $(X, T)$ is {\em transitive} if there exists
some point $x\in X$ whose orbit $\O(x,T)=\{T^nx: n\in \Z\}$ is dense
in $X$ and we call such a point a {\em transitive point}. The system
is {\em minimal} if the orbit of any point is dense in $X$. This
property is equivalent to saying that X and the empty set are the
only closed invariant sets in $X$. $(X,T)$ is {\em topologically
weakly mixing} if the product system $(X\times X, T\times T)$ is
transitive.

\medskip

A {\em factor} of a t.d.s. $(X, T)$ is another t.d.s. $(Y, S)$ such that there exists a continuous and
onto map $\phi: X \rightarrow Y$ satisfying $S\circ \phi = \phi\circ
T$. In this case, $(X,T)$ is called an {\em extension } of $(Y,S)$.
The map $\phi$ is called a {\em factor map}.


\subsubsection{$M(X)$ and $M_T(X)$}

For a t.d.s. $(X,T)$, denote by $M(X)$ the set of all
probability measure on $X$. Let $M_T(X)=\{\mu\in M(X):
T_*\mu=\mu\circ T^{-1}=\mu\}$ be the set of all $T$-invariant
measure of $X$. It is well known that $M_T(X)\neq \emptyset$.

\begin{de}
A t.d.s. $(X,T)$ is called {\em uniquely ergodic} if
there is a unique $T$-invariant probability measure on $X$. It is
called {\em strictly ergodic} if it is uniquely ergodic and minimal.
\end{de}

\subsection{Cubes and faces}

\subsubsection{} Let $X$ be a set, let $d\ge 1$ be an integer, and write
$[d] = \{1, 2,\ldots , d\}$. We view $\{0, 1\}^d$ in one of two
ways, either as a sequence $\ep=\ep_1\ldots \ep_d$ of $0'$s and
$1'$s, or as a subset of $[d]$. A subset $\ep$ corresponds to the
sequence $(\ep_1,\ldots, \ep_d)\in \{0,1\}^d$ such that $i\in \ep$
if and only if $\ep_i = 1$ for $i\in [d]$. For example, ${\bf
0}=(0,0,\ldots,0)\in \{0,1\}^d$ is the same to $\emptyset \subset
[d]$.

Let $V_d=\{0,1\}^d=[d]$ and $V_d^*=V_d\setminus \{{\bf
0}\}=V_d\setminus \{\emptyset\}$.
If ${\bf n} = (n_1,\ldots, n_d)\in \Z^d$ and $\ep\in \{0,1\}^d$, we
define
$${\bf n}\cdot \ep = \sum_{i=1}^d n_i\ep_i .$$
If we consider $\ep$ as $\ep\subset [d]$, then ${\bf n}\cdot \ep  =
\sum_{i\in \ep} n_i .$

\subsubsection{}

We denote $X^{2^d}$ by $X^{[d]}$. A point ${\bf x}\in X^{[d]}$ can
be written in one of two equivalent ways, depending on the context:
$${\bf x} = (x_\ep :\ep\in \{0,1\}^d )= (x_\ep : \ep\subset [d]). $$
Hence $x_\emptyset =x_{\bf 0}$ is the first coordinate of ${\bf x}$.
As examples, points in $X^{[2]}$ are like
$$(x_{00},x_{10},x_{01},x_{11})=(x_{\emptyset}, x_{\{1\}},x_{\{2\}},x_{\{1,2\}}).$$

For $x \in X$, we write $x^{[d]} = (x, x,\ldots , x)\in  X^{[d]}$.
The diagonal of $X^{[d]}$ is $\D^{[d]} = \{x^{[d]}: x\in X\}$.
Usually, when $d=1$, denote diagonal by $\D_X$ or $\D$ instead of
$\D^{[1]}$.

A point ${\bf x} \in X^{[d]}$ can be decomposed as ${\bf x} = ({\bf
x'},{\bf  x''})$ with ${\bf x}', {\bf x}''\in X^{[d-1]}$, where
${\bf x}' = (x_{\ep0} : \ep\in \{0,1\}^{d-1})$ and ${\bf x}''=
(x_{\ep1} : \ep\in \{0,1\}^{d-1})$. We can also isolate the first
coordinate, writing $X^{[d]}_* = X^{2^d-1}$ and then writing a point
${\bf x}\in X^{[d]}$ as ${\bf x} = (x_\emptyset, {\bf x}_*)$, where
${\bf x}_*= (x_\ep : \ep\neq \emptyset) \in X^{[d]}_*$.

\subsubsection{}

The {\em faces} of dimension $r$ of a point in ${\bf x}\in X^{[d]}$
are defined as follows. Let $J\subset [d]$ with $|J| = d-r$ and
$\xi\in \{0,1\}^{d-r}$. The elements $(x_\ep : \ep\in \{0,1\}^{d},
\ep_J=\xi)$ of $X^{[r]}$ are called {\em faces of dimension $r$ } of
${\bf x}$, where $\ep_J = (\ep_i : i \in J)$. Thus any face of
dimension $r$ defines a natural projection from $X^{[d]}$ to
$X^{[r]}$, and we call this the projection along this face.

\subsection{Dynamical parallelepipeds}

\begin{de}
Let $(X, T)$ be a topological dynamical system and let $d\ge 1$ be
an integer. We define $\Q^{[d]}(X)$ to be the closure in $X^{[d]}$
of elements of the form $$(T^{{\bf n}\cdot \ep}x=T^{n_1\ep_1+\ldots
+ n_d\ep_d}x: \ep= (\ep_1,\ldots,\ep_d)\in\{0,1\}^d) ,$$ where ${\bf
n} = (n_1,\ldots , n_d)\in \Z^d$ and $ x\in X$. When there is no
ambiguity, we write $\Q^{[d]}$ instead of $\Q^{[d]}(X)$. An element
of $\Q^{[d]}(X)$ is called a (dynamical) {\em parallelepiped of
dimension $d$}.
\end{de}

As examples, $\Q^{[2]}$ is the closure in $X^{[2]}=X^4$ of the set
$$\{(x, T^mx, T^nx, T^{n+m}x) : x \in X, m, n \in \Z\}$$ and $\Q^{[3]}$
is the closure in $X^{[3]}=X^8$ of the set $$\{(x, T^mx, T^nx,
T^{m+n}x, T^px, T^{m+p}x, T^{n+p}x, T^{m+n+p}x) : x\in X, m, n, p\in
\Z\}.$$

\begin{de}
Let $\phi: X\rightarrow Y$ and $d\in \N$. Define $\phi^{[d]}:
X^{[d]}\rightarrow Y^{[d]}$ by $(\phi^{[d]}{\bf x})_\ep=\phi x_\ep$
for every ${\bf x}\in X^{[d]}$ and every $\ep\subset [d]$.
Let $(X, T)$ be a system and $d\ge 1$ be an integer. The {\em
diagonal transformation} of $X^{[d]}$ is the map $T^{[d]}$.
\end{de}

\begin{de}
{\em Face transformations} are defined inductively as follows: Let
$T^{[0]}=T$, $T^{[1]}_1=\id \times T$. If
$\{T^{[d-1]}_j\}_{j=1}^{d-1}$ is defined already, then set
\begin{equation}\label{def-T[d]}
\begin{split}
T^{[d]}_j&=T^{[d-1]}_j\times T^{[d-1]}_j, \ j\in \{1,2,\ldots, d-1\},\\
T^{[d]}_d&=\id ^{[d-1]}\times T^{[d-1]}.
\end{split}
\end{equation}
\end{de}


The {\em face group} of dimension $d$ is the group $\F^{[d]}(X)$ of
transformations of $X^{[d]}$ spanned by the face transformations.
The {\em parallelepiped group} of dimension $d$ is the group
$\G^{[d]}(X)$ spanned by the diagonal transformation and the face
transformations. We often write $\F^{[d]}$ and $\G^{[d]}$ instead of
$\F^{[d]}(X)$ and $\G^{[d]}(X)$, respectively. For $\G^{[d]}$ and
$\F^{[d]}$, we use similar notations to that used for $X^{[d]}$:
namely, an element of either of these groups is written as $S =
(S_\ep : \ep\in\{0,1\}^d)$. In particular, $\F^{[d]} =\{S\in
\G^{[d]}: S_\emptyset ={\rm id}\}$.

\medskip

For convenience, we denote the orbit closure of ${\bf x}\in X^{[d]}$
under $\F^{[d]}$ by $\overline{\F^{[d]}}({\bf x})$, instead of
$\overline{\O({\bf x}, \F^{[d]})}$.
It is easy to verify that $\Q^{[d]}$ is the closure in $X^{[d]}$ of
$$\{Sx^{[d]} : S\in \F^{[d]}, x\in X\}.$$
If $x$ is a transitive point of $X$, then $\Q^{[d]}$ is the closed
orbit of $x^{[d]}$ under the group $\G^{[d]}$.

\subsection{Measure $\mu^{[k]}$}

\subsubsection{Notation}

When $f_\ep$, $\ep\in V_k=\{0,1\}^d$, are $2^k$ real or complex
valued functions on the set $X$, we define a function
$\bigotimes_{\ep\in V_k} f_\ep$ on $X^{[k]}$ by
\begin{equation*}
    \bigotimes_{\ep\in V_k} f_\ep ({\bf x})=\prod_{\ep\in V_k}
    f_\ep(x_\ep).
\end{equation*}

\subsubsection{}
We define by induction a $T^{[k]}$-invariant measure $\mu^{[k]}$ on
$X^{[k]}$ for every integer $k \ge 0$.

Set $X^{[0]} = X$, $T^{[0]} = T$ and $\mu^{[0]}=\mu$. Assume that
$\mu^{[k]}$ is defined. Let $\I^{[k]}$ denote the
$T^{[k]}$-invariant $\sigma$-algebra of $(X^{[k]}, \mu^{[k]},
T^{[k]})$. Identifying $X^{[k+1]}$ with $X^{[k]}\times X^{[k]}$ as
explained above, we define the system $(X^{[k+1]}, \mu^{[k+1]},
T^{[k+1]})$ to be the relatively independent joining of two copies
of $(X^{[k]}, \mu^{[k]}, T^{[k]})$ over $\I^{[k]}$. That is,
$$\I^{[k]}=\{A\subset X^{[k]}: T^{[k]}A=A\},$$ and
$$\mu^{[k+1]}=\mu^{[k]}\mathop{\times}_{\I^{[k]}} \mu^{[k]}.$$
Equivalently, for all bounded function $f_\ep, \ep\in V_{k+1}$ of
$X$,
\begin{equation}\label{}
    \int_{X^{[k+1]}} \bigotimes_{\ep\in V_{k+1}} f_\ep\
    d\mu^{[k+1]}=\int _{X^{[k]}} \E\Big( \bigotimes_{\eta\in V_k}
    f_{\eta 0}\Big|\I^{[k]}\Big)\E\Big( \bigotimes_{\eta\in V_k}
    f_{\eta 1}\Big|\I^{[k]}\Big)\ d\mu^{[k]}.
\end{equation}

Since $(X, \mu, T)$ is ergodic, $\I^{[0]}$ is the trivial
$\sigma$-algebra and $\mu^{[1]}=\mu\times \mu$. If $(X, \mu, T)$ is
weakly mixing, then by induction $\I^{[k]}$ is trivial and
$\mu^{[k]}$ is the $2^k$ Cartesian power $\mu^{\bigotimes 2^k}$ of
$\mu$ for $k\ge 1$.

We now give an equivalent formulation of the definition of these
measures. For an integer $k\ge 1$, let $(\Omega_k, P_k)$ be the system corresponding to
the $\sigma$-algebra $\I^{[k]}$ and let
\begin{equation}\label{}
    \mu^{[k]}=\int_{\Omega_k} \mu^{[k]}_\omega \ d P_k(\omega)
\end{equation}
denote the ergodic decomposition of $\mu^{[k]}$ under $T^{[k]}$.
Then by definition
\begin{equation}\label{}
\mu^{[k+1]}=\int_{\Omega_k} \mu^{[k]}_\omega\times \mu^{[k]}_\omega
\ d P_k(\omega).
\end{equation}

We generalize this formula. For $k, l \ge 1$, the concatenation of
an element $\a$ of $V_k$ with an element $\b$ of $V_l$ is the
element $\a\b$ of $V_{k+l}$. This defines a bijection of $V_k \times
V_l$ onto $V_{k+l}$ and gives the identification
$(X^{[k]})^{[l]}=X^{[k+1]}$. By \cite[Lemma 3.1.]{HK05}
\begin{equation}\label{}
    \mu^{[k+l]}=\int_{\Omega_k}(\mu^{[k]}_\omega)^{[l]}\
    dP_k(\omega).
\end{equation}

\subsection{Characteristic factors $(Z_k,\mu_k)$}

\subsubsection{}
Notice that in \cite{HK05}, $\G^{k}$ and $\F^{[k]}$ are denoted by
$\mathcal{T}_{k-1}^{[k]}$ and $\mathcal{T}_*^{[k]}$ respectively.
Let $\J^{[k]}$ denote the $\sigma$-algebra of sets on $X^{[k]}$ that
are invariant under the group $\F^{[k]}$. On $(X^{[k]},\mu^{[k]})$,
the $\sigma$-algebra $\J^{[k]}$ coincides with the $\sigma$-algebra
of sets depending only on the coordinate ${\bf 0}$
(\cite[Proposition 3.4]{HK05}).

\begin{prop}\cite{HK05}
For all $k\in \N$, $(X^{[k]},\mu^{[k]})$ is ergodic for the group of
side transformations $\G^{[d]}$. And $(\Omega_k, P_k)$ is ergodic
under the action of the group $\F^{[k]}$.
\end{prop}

We consider the $2^k-1$-dimensional marginals of $\mu^{[k]}$. Recall
that $V^*_k = V_k \setminus \{{\bf 0}\}$. Consider a point ${\bf
x}\in X^{[k]}$ as a pair $(x_{\bf 0}, {\bf x_*})$, with $x_{\bf
0}\in X$ and ${\bf x_*}\in X^{[k]}_*$. Let $\mu^{[k]}_*$ denote the
measure on $X^{[k]}_*$, which is the image of $\mu^{[k]}$ under the
natural projection ${\bf x}\mapsto {\bf x_*}$ from $X^{[k]}$ onto
$X^{[k]}_*$.

All the transformations belonging to $\G^{[k]}$ factor through the
projection $X^{[k]}\rightarrow X^{[k]}_*$ and induce transformations
of $X^{[k]}_*$ preserving $\mu^{[k]}_*$. This defines a
measure-preserving action of the group $\G^{[k]}$ and of its
subgroup $\F^{[k]}$ on $X^{[k]}_*$. The measure $\mu^{[k]}_*$ is
ergodic for the action of $\G^{[k]}$.

On the other hand, all the transformations belonging to $\G^{[k]}$
factor through the projection ${\bf x}\mapsto {x_{\bf 0}}$ from
$X^{[k]}$ to $X$, and induce measure-preserving transformations of
$X$. The transformation $T^{[k]}$ induces the transformation $T$ on
$X$, and each transformation belonging to $\F^{[k]}$ induces the
trivial transformation on $X$. This defines a measure-preserving
ergodic action of the group $\G^{[k]}$ on $X$, with a trivial
restriction to the subgroup $\F^{[k]}$.

\subsubsection{A system of order $k$}
Let $ \J^{[k]}_*$ denote the $\sigma$-algebra of subsets of
$X^{[k]}_*$ which are invariant under the action of $\F^{[k]}$.
Since the $\sigma$-algebra $\J^{[k]}$ coincides with the
$\sigma$-algebra of sets depending only on the coordinate ${\bf 0}$
(\cite[Proposition 3.4]{HK05}). Hence there exists a
$\sigma$-algebra $\ZZ_{k-1}$ of $X$ such that $\ZZ_{k-1}$ is
isomorphic to $\J^{[k]}_*$. To be precise, for each $A\in
\J^{[k]}_*$, there is unique $B\in \ZZ_{k-1}$ such that $1_B(x_{\bf
0})=1_{A}({\bf x_*})$ for $\mu^{[k]}$-almost every ${\bf x}=(x_{\bf
0}, {\bf x_*})\in X^{[k]}$.

\begin{de}
The $\sigma$-algebra $\ZZ_{k}$ is invariant under $T$ and so defines
a factor of $(X,\mu, T)$ written $(Z_{k}(X), \mu_{k}, T)$, or simply
$(Z_{k}, \mu_{k}, T)$. The factor map $X \rightarrow Z_{k}$ is
written by $\pi_{k}$.

$(Z_{k}, \ZZ_k, \mu_k, T)$ is called a {\em system of order $k$}.
\end{de}

$(Z_k,\ZZ_k, \mu_k)$ has a very nice structure:

\begin{thm}\cite{HK05}
Let $(X,\X,\mu , T)$ be an ergodic system and $k\in \N$. Then the
system $(Z_k, \ZZ_k, \mu_k, T)$ is a (measure theoretic) inverse
limit of $k$-step nilsystems.
\end{thm}

\begin{rem}
In this section we follows from the treatment of Host and Kra.
Ziegler has a different approach, see \cite{Z}. For more details
about the difference between these two methods, see Leibman's notes
in the appendix in \cite{Bergelson06}.
\end{rem}

\subsubsection{Properties about $Z_k$} The following properties may
be useful in the next section.

\begin{thm}\cite{HK05, HK09}\label{HK-Zk}
Let $k\ge 2$ is an integer and $(X =G/ \Gamma, \mu, T )$ be an
ergodic $(k-1)$-step nilsystem.
\begin{enumerate}
  \item The measure $\mu^{[k]}$ is the Haar measure of a sub-nilmanifold
  $X_k=\Q^{[k]}$ of $X^{[k]}$. $(\Q^{[k]}, \mu^{[k]}, \G^{[k]})$ is
  strictly ergodic.

  \item Let $X_{k*}$ be the image of $X_k$ under the projection ${\bf x}\mapsto
  {\bf x_*}$ from $X^{[k]}$ to $X^{[k]}_*=X^{2^k-1}$.
  There exists a smooth map $\Phi: X_{k*}\rightarrow X_k$ such
  that $$X_k=\{(\Phi({\bf x_*}),{\bf x_*}): {\bf x}\in X_{k*}\}.$$

  \item For every $x\in X$, let
  $W_{k,x} = \{{\bf x}\in X_k: x_{\bf 0} = x\}$. Then
  $W_{k,x}=\overline{\F^{[k]}}(x^{[k]})$ and it is
  uniquely ergodic under $\F^{[k]}$.

  \item For every $x\in  X$, let $\rho_{k,x}$ be the invariant measure of $W_{k,x}$.
  Then for every $x\in X$ and $g \in G$, $\rho_{k,gx}$
  is the image of $\rho_{k,x}$ under the translation by $g^{[k]} = (g, g,\ldots,g)$.
\end{enumerate}
\end{thm}

We need the following result replacing $d$-step nilmanifold with $d$-step nilsystem.

\begin{thm}
Let $k\ge 2$ is an integer and $(X,T,\mu)$ is an
ergodic $(k-1)$-step nilsystem.
\begin{enumerate}
  \item The measure $\mu^{[k]}$ is an invariant measure of
  $\Q^{[k]}$. $(\Q^{[k]}, \mu^{[k]}, \G^{[k]})$ is
  strictly ergodic.

  \item For every $x\in X$, let
  $W_{k,x} = \{{\bf x}\in \Q^{[k]}: x_{\bf 0} = x\}$. Then
  $W_{k,x}=\overline{\F^{[k]}}(x^{[k]})$ and it is
  uniquely ergodic under $\F^{[k]}$.

  \item For every $x\in  X$, let $\rho_{k,x}$ be the invariant measure of $W_{k,x}$.
  Then for every $x\in X$, $\rho_{k,Tx}$
  is the image of $\rho_{k,x}$ under the translation by $T^{[k]} = (T, T,\ldots,T)$.
\end{enumerate}
\end{thm}
\begin{proof} By \cite{HK05} $(X,T, \mu)$ is an inverse limit of $(X_j =G_j/ \Gamma_j, \mu_j, T)$
of $d$-step nilsystems. Then the result follows.
\end{proof}

\section{Deducing Theorems C and D from Theorems A and B}\label{sec-model}

In this section we show how we obtain Theorem C (resp. D) from Theorem A (resp. B). The proof of Theorem A
will be carried out in the next section and the proof of Theorem B will be presented in Section 5.
Moreover, we will use Furstenberg-Weiss' almost one-to-one Theorem
to get a $d$-step almost automorphic model.

\subsection{The proof of Theorem D assuming Theorem B}
To simplify some statements, we introduce the following definition. Recall that $\tau_d=T\times
\ldots\times T \ (d \ \text{times})$, $\sigma_d =T\times \ldots \times T^{d}$
and $\langle\tau_d, \sigma_d\rangle$ is the group generated by
$\tau_d$ and $\sigma_d$. Moreover,
$N_d(\h{X})=\overline{\O(\Delta_d(\h{X}),
\sigma_d)}=\overline{\O((x,\ldots,x), \langle\tau_d, \sigma_d\rangle)}$ when $(\h{X},T)$ is minimal.

\begin{de}
Let $(X,\X, \mu, T)$ be an ergodic m.p.t. and $(\h{X}, T)$
be its model.
\begin{enumerate}
\item For $d\in \N$, $(\h{X}, T)$ is called an {\em
$\F^{[d]}$-strictly ergodic model} for $(X,\X, \mu, T)$ if $(\h{X},
T)$ is a strictly ergodic model and
$(\overline{\F^{[d]}}(x^{[d]}),\F^{[d]})$ is strictly ergodic for
all $x\in \hat{X}$.

\item For $d\in \N$, $(\h{X}, T)$ is called a {\em $\G^{[d]}$-strictly
ergodic model} for $(X,\X, \mu, T)$ if $(\h{X}, T)$ is a strictly
ergodic model and $(\Q^{[d]},\G^{[d]})$ is strictly ergodic.

\item For $d\in \N$, $(\h{X}, T)$ is called a
{\em $\langle\tau_d, \sigma_d\rangle-$strictly ergodic model} for $(X,\X, \mu,
T)$ if $(\h{X}, T)$ is a strictly ergodic model and $(N_d(\h{X}),
\langle\tau_d, \sigma_d\rangle)$ is strictly ergodic.
\end{enumerate}
\end{de}

To obtain the connection between Theorems A (resp. B) and C (resp. D), we need the following formula which is easy to be verified.

\begin{lem}\label{lem-product}
Let $\{a_i\}, \{b_i\}\subseteq \C$. Then
\begin{equation}\label{}
\prod_{i=1}^k a_i-\prod_{i=1}^k b_i=(a_1-b_1)b_2\ldots b_k+
a_1(a_2-b_2)b_3\ldots b_k +a_1\ldots a_{k-1}(a_k-b_k).
\end{equation}
\end{lem}

\medskip
We will show Theorem D can be deduced from Theorem B. The proofs of Theorem C assuming Theorem A follows similarly.

\noindent{\bf The proof of Theorem D assuming Theorem B:}
Since $(X,\X, \mu, T)$ has a $\langle \tau_d,\sigma_d\rangle-$strictly ergodic model, we
may assume that $(X,T)$ itself is a minimal t.d.s. and
$\mu$ is its unique measure such that $(N_d(X), \langle\tau_d, \sigma_d
\rangle)$ is uniquely ergodic with the unique measure $\ll_{\tau,\sigma;d}$ defined in (\ref{limit-D}).

Let $\delta>0$. Without loss of generality, we assume that for all $1\le j\le d$,
$\|f_j\|_\infty\le 1$. Choose continuous functions $g_j$ such that
$\|g_j\|_\infty\le 1$ and $\|f_j-g_j\|_1<\d/d$ for all $1\le j\le
d$. We have
\begin{equation}\label{yexd1}
\begin{split}
&  \left | \frac{1}{N^2} \sum_{n\in [0,N-1] \atop{m\in [0,N-1]}} \prod_{j=1}^{d} f_j(T^{n+(j-1)m}x) -
\int_{N(X)}\otimes_{j=1}^df_jd\ll_{\tau,\sigma;d} \right |\\
& \le  \left | \frac{1}{N^2} \sum_{n\in [0,N-1] \atop{m\in [0,N-1]}} \prod_{j=1}^{d} f_j(T^{n+(j-1)m}x) -
\frac{1}{N^2} \sum_{n\in [0,N-1] \atop{m\in [0,N-1]}}\prod_{j=1}^d g_j(T^{n+(j-1)m}x)\right |\\
&+\left | \frac{1}{N^2} \sum_{n\in [0,N-1] \atop{m\in [0,N-1]}}\prod_{j=1}^d g_j(T^{n+(j-1)m}x)-
\int_{N(X)}\otimes_{j=1}^dg_jd\ll_{\tau,\sigma;d} \right |\\
&+\left | \int_{N(X)}\otimes_{j=1}^dg_jd\ll_{\tau,\sigma;d}-
\int_{N(X)}\otimes_{j=1}^df_jd\ll_{\tau,\sigma;d} \right |.
\end{split}
\end{equation}

Now by Pointwise Ergodic Theorem for $\Z^2$ i.e. Theorem
\ref{Lindenstrauss} (applying to $(n,m)\mapsto T^{n+(j-1)m}$) we have that for all $1\le j\le d$
\begin{equation*}
    \frac {1}{N^2} \sum_{n\in [0,N-1] \atop{m\in
[0,N-1]}}\Big | f_j(T^{n+(j-1)m}x)-g_j(T^{n+(j-1)m}x)\Big | \lra
\|f_j-g_j\|_1,\ a.e. \quad N\to \infty.
\end{equation*}
Hence by Lemma \ref{lem-product}, when $N$ is large
\begin{equation}\label{yexd2}
\begin{split}
&   \left | \frac{1}{N^2} \sum_{n\in [0,N-1] \atop{m\in [0,N-1]}}
   \prod_{j=1}^{d} f_j(T^{n+(j-1)m}x) -
\frac{1}{N^2} \sum_{n\in [0,N-1] \atop{m\in [0,N-1]}}
   \prod_{j=1}^d g_j(T^{n+(j-1)m}x)\right |\\
& \le \sum_{j=1}^d \Big[\frac {1}{N^2} \sum_{n\in [0,N-1]
\atop{m\in
[0,N-1]}}\Big | f_j(T^{n+(j-1)m}x)-g_j(T^{n+(j-1)m}x)\Big |\Big ]\\
& <2\sum_{j=1}^d \|f_j-g_j\|_1\le 2\d,\ a.e.
\end{split}
\end{equation}

Note that
\begin{equation*}
\begin{split}
&\frac{1}{N^2} \sum_{n\in [0,N-1] \atop{m\in [0,N-1]}}
   \prod_{j=1}^d g_j(T^{n+(j-1)m}x)\\
   =& \frac{1}{N^2} \sum_{n\in [0,N-1] \atop{m\in [0,N-1]}}
    g_1(T^nx)g_2(T^{n+m}x)\ldots g_d(T^{n+(d-1)m}x) \\ =&
    \frac{1}{N^2} \sum_{n\in [0,N-1] \atop{m\in [0,N-1]}}
    g_1\otimes \ldots \otimes g_d \Big( \tau_d^n\sigma_d^m(x,x,\ldots,x)\Big).
\end{split}
\end{equation*}

Since $g_1\otimes \ldots \otimes g_d: X^{d}\rightarrow \R$ is
continuous and $(N_d(X), \langle \tau_d, \sigma_d \rangle)$ is uniquely ergodic,
by Theorem \ref{unique-ergodic}, $\displaystyle \frac{1}{N^2}
\sum_{n\in [0,N-1] \atop{m\in [0,N-1]}}
   \prod_{j=1}^d g_j(T^{n+(j-1)m}x)$ converges pointwisely to $\int_{N(X)}\otimes_{j=1}^dg_jd\ll_{\tau,\sigma;d} $.
So when $N$ is large
\begin{equation}\label{xdye3}
\left | \frac{1}{N^2} \sum_{n\in [0,N-1] \atop{m\in [0,N-1]}}\prod_{j=1}^d g_j(T^{n+(j-1)m}x)-
\int_{N(X)}\otimes_{j=1}^dg_jd\ll_{\tau,\sigma;d} \right |\le \d.
\end{equation}

By Lemma \ref{lem-product},
\begin{equation}\label{yexd4}
\begin{split}
& \left | \int_{N(X)}\otimes_{j=1}^dg_jd\ll_{\tau,\sigma;d}-
\int_{N(X)}\otimes_{j=1}^df_jd\ll_{\tau,\sigma;d} \right |\\
&\le \sum_{j=1}^d\int_{N(X)}|g_j-f_j|d\ll_{\tau,\sigma;d} \le \d.
\end{split}
\end{equation}

So combining (\ref{yexd1})-(\ref{yexd4}), when $N$ is large, we have
$$ \left | \frac{1}{N^2} \sum_{n\in [0,N-1] \atop{m\in [0,N-1]}} \prod_{j=1}^{d} f_j(T^{n+(j-1)m}x) -
\int_{N(X)}\otimes_{j=1}^df_jd\ll_{\tau,\sigma;d} \right |\le 4\d, a.e. $$

This clearly implies that

\begin{equation*}
    \frac{1}{N^2} \sum_{n\in [0,N-1] \atop{m\in [0,N-1]}}
    f_1(T^nx)f_2(T^{n+m}x)\ldots f_d(T^{n+(d-1)m}x)
\end{equation*}
converge $\mu$ a.e.. The proof is completed.


\medskip



\begin{rem}
It is easy to see that if  (\ref{C1}) holds for all $d$, then
we have (\ref{C2}) holds for all $d$. That is, (\ref{C1}) is more
fundamental. For example, if we want to get $\G^{[1]}$-case:
\begin{equation*}
    \frac{1}{N^2}\sum_{0\le n_1,n_2\le N-1}
    f_{0}(T^{n_1}x)f_{1}(T^{n_1+n_2}x),
\end{equation*}
then what need do is in the $\F^{[2]}$-case
\begin{equation*}
    \frac{1}{N^2}\sum_{0\le n_1,n_2\le N-1} f_{01}(T^{n_1}x)f_{10}(T^{n_2}x)f_{11}(T^{n_1+n_2}x)
\end{equation*}
by setting $f_{00}=f_0, f_{10}=1$ and $f_{11}=f_1$.
\end{rem}

\subsection{$d$-step almost automorphic systems}
$d$-step almost automorphic systems were defined and studied in \cite{HSY} which
are the generalization of Veech's almost automorphic systems.
\begin{de}
Let $(X,T)$ be a minimal t.d.s. and $d\in \N$.
$(X,T)$ is called a {\em $d$-step almost automorphic} system if it
is an almost one-to-one extension of a $d$-step nilsystem.
\end{de}

See \cite{HSY} for more discussion about $d$-step almost automorphy. In this
subsection we will show that in Theorem A we can also require the models are $d$-step almost automorphic systems.
To do so, first we state Furstenberg-Weiss's almost one-to-one Theorem.


\begin{thm}[Furstenberg-Weiss]\cite{FW89}\label{Furstenberg-Weiss}
Let $(Y, T)$ be a non-periodic minimal t.d.s., and let
$\pi' : X' \rightarrow Y$ be an extension of $(Y, T)$ with $(X', T)$
topologically transitive and $X'$ a compact metric space.
\[
\begin{CD}
X' @>{\theta}>> X \\
@V{\pi'}VV    @VV{\pi}V\\
Y @>{ }>> Y
\end{CD}
\]
Then there exists an almost 1-1 minimal extension
$\pi:(X,T)\rightarrow (Y, T)$, a Borel subset $X_0'\subseteq X'$ and
a Borel measurable map $\theta: X_0'\rightarrow X$ satisfying:
\begin{enumerate}
  \item $\theta \circ T=T\circ\theta$;
  \item $\pi\circ \theta=\pi'$;
  \item $\theta$ is a Borel isomorphism of $X_0'$ onto its image
  $X_0=\theta(X_0')\subseteq X$;
  \item $\mu(X_0')=1$ for any $T$-invariant measure $\mu$ on $X'$.
  \item if $(X', T)$ is uniquely ergodic, then $(X,T)$ can be chosen
  to be uniquely (hence strictly) ergodic.
\end{enumerate}
\end{thm}

\begin{rem}
In \cite[Theorem 1]{FW89}, (1)-(4) are stated. From the proof of the
theorem given in \cite{FW89}, we have (5), which is pointed out in
\cite{GW94}.
\end{rem}



Let $(X,\X, \mu, T)$ be an ergodic system with non-trivial
nil-factors (non-triviality here means infinity) and $d\in \N$. Let
$\pi_d: X\rightarrow Z_d$ be the factor map from $X$ to its $d$-step
nilfactor $Z_d$. By definition, $Z_d$ may be regarded as a
t.d.s. in the natural way. By Weiss's  theorem \cite{Weiss85},
there is a uniquely ergodic model $(\h{X'}, \h{\X'}, \h{\mu}, T)$
for $(X,\X, \mu, T)$ and a factor map $\h{\pi_d'}: \h{X'}\rightarrow
Z_d$ which is a model for $\pi_d: X\rightarrow Z_d$.
\[
\begin{CD}
X @>{\phi}>> \h{X'} @>{\theta}>> \h{X}\\
@V{\pi_d}VV   @VV{\h{\pi_d'}}V   @VV{\h{\pi_d}}V\\
Z_d @>{ }>> Z_d @>{ }>> Z_d
\end{CD}
\]
Now by  Theorem \ref{Furstenberg-Weiss},
$\h{\pi_d'}: \h{X'}\rightarrow Z_d$ may be replaced by $\h{\pi_d}:
\h{X}\rightarrow Z_d$, where $\h{\pi_d}$ is almost 1-1 and $\h{X'}$
and $\h{X}$ are measure theoretically isomorphic. In particular,
$(\h{X},T)$ is a strictly ergodic model for $(X,\X,\mu, T)$.

As we described in the introduction, one once we have a model $\h{\pi}:\h{X}\lra Z_d$ then
it is  $\F^{[d]}$ and $\G^{[d]}$ models. Hence combining above discussion with Theorem A, we have
\begin{thm}\label{AA-model}
Let $d\in \N$. Then every ergodic m.p.t. with a non-trivial
$d$-step nilfactor has an $\F^{[d]}$ and $\G^{[d]}$ strictly ergodic
model $(X,T)$ which is a $d$-step almost automorphic system.
\end{thm}

\section{Proof of Theorem A}\label{sec-proof}

In this section we give a proof for Theorem A. To make the
idea of the proof clearer before going into the proof for the general case we show
the cases when $d=1$ and $d=2$ first. We also give a proof for weakly mixing systems for
independent interest. Finally we show the general case by induction.

\subsection{Case when $d=1$}\label{d=1}
By Jewett-Krieger's Theorem, every ergodic
system has a strictly ergodic model. Now we show this model is
$\F^{[1]}$-strictly ergodic. Let $(X,T)$ be a strictly ergodic
system and let $\mu$ be its unique $T$-invariant measure. Note that
$\F^{[1]}=\langle\id \times T\rangle$. Hence for all $x\in X$,
$$\overline{\F^{[1]}}(x^{[1]})=\{x\}\times X.$$
Since $(X,T)$ is uniquely ergodic, $\d_x\times \mu$ is the only
$\F^{[1]}$-invariant measure of $\overline{\F^{[1]}}(x^{[1]})$. In
this case Theorem A(1) is nothing but Birkhorff pointwise
ergodic theorem.

\medskip

Now consider $\Q^{[1]}$. Since $\G^{[1]}=\langle T\times T, \id\times T\rangle$,
it is easy to see that $\Q^{[1]}=X\times X$. Let $\ll$ be a
$\G^{[1]}$-invariant measure of $(X^{[1]}, \X^{[1]})=(X\times X,
\X\times \X)$. Since $\ll$ is $T\times T$-invariant, it is a
self-joining of $(X,\X,\mu,T)$ and has $\mu$ as its marginal. Let
\begin{equation}\label{a2}
\ll =\int_X \d_x\times \ll_x\ d \mu(x)
\end{equation}
be the disintegration of $\ll$ over $\mu$. Since $\ll$ is $\id
\times T$-invariant, we have
\begin{equation*}
    \ll=\id \times T \ll=\int_X \d_x\times T\ll_x \ d \mu(x).
\end{equation*}
The uniqueness of disintegration implies that
\begin{equation*}
    T\ll_x=\ll_x, \mu \ a.e.
\end{equation*}
Since $(X,\X,T)$ is uniquely ergodic, $\ll_x=\mu, \ \mu\ $ a.e. Thus
by (\ref{a2}) one has that
$$\ll=\int_X \d_x\times \ll_x\ d\mu(x)=\int_X \d_x\times \mu \ d\mu(x)=\mu\times \mu.$$
Hence $(\Q^{[1]},\G^{[1]})$ is uniquely ergodic, and
$\mu^{[1]}=\mu\times \mu$ is its unique $\G^{[1]}$-invariant
measure.


\subsection{Weakly mixing systems}\label{sec-wm}

In this subsection we show Theorem A holds for weakly
mixing systems. This result relies on the following proposition.

\begin{prop}\label{prop-wm-Fd}
Let $(X,T)$ be uniquely ergodic, $(X,\X, \mu, T)$ be weakly mixing and $d\in
\N$. Then
\begin{enumerate}
  \item $(X^{[d]},  \G^{[d]})$ is uniquely
  ergodic with the unique measure $\mu^{[d]}= \underbrace{\mu\times \ldots \times \mu}_{2^d\
  \text{times}}$.
  \item $(X^{[d]}_*, \F^{[d]})$ is uniquely
  ergodic with the unique measure $\mu^{[d]}_*= \underbrace{\mu\times \ldots \times \mu}_{2^d-1\ \text{times}}$.
\end{enumerate}
\end{prop}

\begin{proof}
We prove the result inductively. First we show the case when $d=1$.
In this case $\F^{[1]}=\langle\id \times T\rangle$ and $\G^{[1]}=\langle\id \times T,
T\times T\rangle$. Hence $(X^{[1]}_*, \X^{[1]}_*, \F^{[1]}_*)=(X,\X,T)$,
and it follows that $\mu^{[1]}_*=\mu$ is the unique $T$-invariant
measure. Let $\ll$ be a $\G^{[1]}$-invariant measure of $(X^{[1]},
\X^{[1]})=(X\times X, \X\times \X)$. By the argument in subsection~
\ref{d=1}, we know that $\ll=\mu^{[1]}=\mu\times \mu$.

Now assume the statements hold for $d-1$, and we show the case for
$d$. Let $\ll$ be a $\G^{[d]}$-invariant measure of $(X^{[d]},
\X^{[d]})$. Let $$p_1: (X^{[d]}, \G^{[d]})\rightarrow (X^{[d-1]},
\G^{[d-1]}); \ {\bf x} =({\bf x'}, {\bf x''})\mapsto {\bf x'}$$
$$p_2: (X^{[d]}, \G^{[d]})\rightarrow (X^{[d-1]},
\G^{[d-1]}); \ {\bf x} =({\bf x'}, {\bf x''})\mapsto {\bf x''}$$  be
the projections. Then $(p_2)_*(\ll)$ is a $\G^{[d-1]}$-invariant
measure of $X^{[d-1]}$. By inductive assumption,
$(p_2)_*(\ll)=\mu^{[d-1]}$.  Let
\begin{equation}\label{a3}
\ll =\int_{X^{[d-1]}} \ll_{{\bf x}}\times \d_{\bf x}\ d
\mu^{[d-1]}({\bf x})
\end{equation}
be the disintegration of $\ll$ over $\mu^{[d-1]}$. Since $\ll$ is
$T^{[d]}_d=\id^{[d-1]} \times T^{[d-1]}$-invariant, we have
\begin{eqnarray*}
    \ll &= &\id^{[d-1]} \times T^{[d-1]} \ll=
    \int_{X^{[d-1]}} \ll_{\bf x}\times T^{[d-1]}\d_{\bf x} \ d \mu^{[d-1]}({\bf x})\
    \\ & = & \int_{X^{[d-1]}} \ll_{\bf x}\times \d_{T^{[d-1]}\bf x} \ d \mu^{[d-1]}({\bf x})\\
    \\&=& \int_{X^{[d-1]}} \ll_{(T^{[d-1]})^{-1}\bf x}\times \d_{\bf x} \ d \mu^{[d-1]}({\bf x}).
\end{eqnarray*}
The uniqueness of disintegration implies that
\begin{equation}\label{a4}
    \ll_{(T^{[d-1]})^{-1}\bf x}=\ll_{\bf x},\quad \mu^{[d-1]} \ a.e. \ {\bf x}\in X^{[d-1]}.
\end{equation}

Define $$F: (X^{[d-1]},\X^{[d-1]},T^{[d-1]}) \lra M(X^{[d-1]}): \
{\bf x}\mapsto \ll_{\bf x}.$$ By (\ref{a4}), $F$ is a
$T^{[d-1]}$-invariant $M(X^{[d-1]})$-value function. Since
$(X,\X,\mu, T)$ is weakly mixing, $(X^{[d-1]},\X^{[d-1]},T^{[d-1]})$
is ergodic and hence $\ll_{\bf x}=\nu, \ \mu^{[d-1]}\ $ a.e. for
some $\nu\in M(X^{[d-1]})$. Thus by (\ref{a3}) one has that
$$\ll=\int_{X^{[d-1]}}  \ll_{\bf x}\times \d_{\bf x}\ d\mu^{[d-1]}({\bf x})=
\int_{X^{[d-1]}} \nu \times\d_{\bf x}\ d\mu^{[d-1]}({\bf x})=
\nu\times\mu^{[d-1]} .$$ Then we have that $\nu=(p_1)_*(\ll)$ is a
$\G^{[d-1]}$-invariant measure of $X^{[d-1]}$. By inductive
assumption, $\mu^{[d-1]}$ is the only $\G^{[d-1]}$-invariant measure of $X^{[d-1]}$ and hence $\nu=(p_1)_*(\ll)=\mu^{[d-1]}$. Thus
$\ll=\mu^{[d-1]}\times \mu^{[d-1]}=\mu^{[d]}$. That is, $(X^{[d]},
\X^{[d]}, \mu^{[d]}, \G^{[d]})$ is uniquely ergodic.

\medskip

Now we show that $(X^{[d]}_*, \X^{[d]}_*, \mu^{[d]}_*, \F^{[d]})$ is
uniquely ergodic. The proof is similar. Let $\ll$ be a $\F^{[d]}$-invariant measure of
$(X^{[d]}_*, \X^{[d]}_*)$. Let $$q_1: (X^{[d]}_*,
\F^{[d]})\rightarrow (X^{[d-1]}_*, \F^{[d-1]}); \ {\bf x} =({\bf
x'_*}, {\bf x''})\mapsto {\bf x'_*}$$
$$q_2: (X^{[d]}, \F^{[d]})\rightarrow (X^{[d-1]},
\G^{[d-1]}); \ {\bf x} =({\bf x'_*}, {\bf x''})\mapsto {\bf x''}$$
be the projections. Then $(q_2)_*(\ll)$ is a $\G^{[d-1]}$-invariant
measure of $X^{[d-1]}$. By inductive assumption,
$(q_2)_*(\ll)=\mu^{[d-1]}$.  Let
\begin{equation}\label{a7}
\ll =\int_{X^{[d-1]}} \ll_{{\bf x}}\times \d_{\bf x}\ d
\mu^{[d-1]}({\bf x})
\end{equation}
be the disintegration of $\ll$ over $\mu^{[d-1]}$. Since $\ll$ is
$T^{[d]}_d=\id^{[d-1]} \times T^{[d-1]}$-invariant, we have
\begin{eqnarray*}
    \ll &= &\id^{[d-1]} \times T^{[d-1]} \ll=
    \int_{X^{[d-1]}} \ll_{\bf x}\times T^{[d-1]}\d_{\bf x} \ d \mu^{[d-1]}({\bf x})\
    \\ & = & \int_{X^{[d-1]}} \ll_{\bf x}\times \d_{T^{[d-1]}\bf x} \ d \mu^{[d-1]}({\bf x})\\
    \\&=& \int_{X^{[d-1]}} \ll_{(T^{[d-1]})^{-1}\bf x}\times \d_{\bf x} \ d \mu^{[d-1]}({\bf x}).
\end{eqnarray*}
The uniqueness of disintegration implies that
\begin{equation}\label{a8}
    \ll_{(T^{[d-1]})^{-1}\bf x}=\ll_{\bf x},\quad \mu^{[d-1]} \ a.e.
\end{equation}

Define $$F: (X^{[d-1]},\X^{[d-1]},T^{[d-1]}) \lra M(X^{[d-1]}_*): \
{\bf x}\mapsto \ll_{\bf x}.$$ By (\ref{a8}), $F$ is a
$T^{[d-1]}$-invariant $M(X^{[d-1]}_*)$-value function. Since
$(X,\X,\mu, T)$ is weakly mixing, $(X^{[d-1]},\X^{[d-1]},T^{[d-1]})$
is ergodic and hence $\ll_{\bf x}=\nu, \ \mu^{[d-1]}\ $ a.e. for
some $\nu\in M(X^{[d-1]}_*)$. Thus by (\ref{a7}) one has that
$$\ll=\int_{X^{[d-1]}}  \ll_{\bf x}\times \d_{\bf x}\ d\mu^{[d-1]}({\bf x})=
\int_{X^{[d-1]}} \nu \times\d_{\bf x}\ d\mu^{[d-1]}({\bf x})=
\nu\times\mu^{[d-1]} .$$ Then we have that $\nu=(q_1)_*(\ll)$ is a
$\F^{[d-1]}$-invariant measure of $X^{[d-1]}_*$. By inductive
assumption, $\mu^{[d-1]}_*$ is the only $\F^{[d-1]}$-invariant measure of $X^{[d-1]}_*$ and $\nu=(q_1)_*(\ll)=\mu^{[d-1]}_*$. Thus
$\ll=\mu^{[d-1]}_*\times \mu^{[d-1]}=\mu^{[d]}_*$. Hence
$(X_*^{[d]}, \X^{[d]}_*, \mu^{[d]}_*, \F^{[d]})$ is uniquely
ergodic. The proof is completed.
\end{proof}

\begin{thm}\label{Thm-wm-Fd}
If $(X,\X, \mu, T)$ is a weakly mixing m.p.t., then it has
an $\F^{[d]}$ and $\G^{[d]}$ strictly ergodic model for all $d\in
\N$.
\end{thm}

\begin{proof}
By Jewett-Krieger' Theorem, $(X,\X,
\mu, T)$ has a uniquely ergodic model. Without loss of generality,
we assume that $(X,T)$ itself is a  minimal t.d.s.
and $\mu$ is its unique $T$-invariant measure. By \cite[Theorem
3.11.]{SY}, $(\Q^{[d]}=X^{[d]}, \G^{[d]})$ is minimal, and for all
$x\in X$, $(\overline{\F^{[d]}}(x^{[d]}), \F^{[d]})$ is minimal and
$\overline{\F^{[d]}}(x^{[d]})=\{x\}\times X^{[d]}_*=\{x\}\times
X^{2^d-1}.$ By Proposition \ref{prop-wm-Fd}, $(\Q^{[d]}, \G^{[d]})$
and $(\overline{\F^{[d]}}(x^{[d]}),\F^{[d]})$ (for all $x\in X$) are
uniquely ergodic . Hence it has an $\F^{[d]}$ and $\G^{[d]}$
strictly ergodic model for all $d\in \N$.
\end{proof}

\subsection{Case when $d=2$}\label{d=2}
In this case we can give the explicit description of the
unique measure. Since the proof is long, we put it in Appendix \ref{sec-H}.
People familiar with the materials can read the proof for the general
case directly.

\subsection{General case}

In this section we prove Theorem A in the general case. We prove it by
induction on $d$. $d=1$ and $d=2$ is showed in subsection \ref{d=1} and
Appendix \ref{sec-H}. Now we assume $d$ and show the case when $d+1$.

\subsubsection{Notations} Recall that $\I^{[d]}$ is the
$T^{[d]}$-invariant $\sigma$-algebra of $(X^{[d]}, \mu^{[d]},
T^{[d]})$ and
$$\mu^{[d+1]}=\mu^{[d]}\mathop{\times}_{\I^{[d]}} \mu^{[d]}.$$
Let
\begin{equation}\label{Omega-d}
\begin{split}
    (X^{[d]}, \mu^{[d]})\stackrel{\phi}{\lra} (\Omega_d,
    \I^{[d]}, P_d);\  {\bf x}&\lra \phi({\bf x})
    \end{split}
\end{equation}
be the factor map. Let
\begin{equation}\label{}
    \mu^{[d]}=\int_{\Omega_d} \mu^{[d]}_\omega \ d P_d(\omega)
\end{equation}
denote the ergodic decomposition of $\mu^{[d]}$ under $T^{[d]}$.
Then by definition
\begin{equation}\label{}
\mu^{[d+1]}=\int_{\Omega_d} \mu^{[d]}_\omega\times \mu^{[d]}_\omega
\ d P_d(\omega).
\end{equation}

\subsubsection{A property about $Z_d$}

\begin{prop}\cite[Proposition 4.7.]{HK05}\label{HK-Pro4.7}
Let $d\ge 1$ be an integer.
\begin{enumerate}
  \item As a joining of $2^d$ copies of $(X, \mu)$, $(X^{[d]},\mu^{[d]})$ is
relatively independent over the joining $(Z_{d-1}^{[d]},
\mu_{d-1}^{[d]})$ of $2^d$ copies of $(Z_{d-1}, \mu_{d-1})$.

  \item $Z_d$ is the smallest factor $Y$ of $X$ so that the $\sigma$-algebra
$\I^{[d]}$ is measurable with respect to $Y^{[d]}$.
\end{enumerate}
\end{prop}

We say that a factor map $\pi: (X,\X,\mu, T)\rightarrow
(Y,\Y,\nu,T)$ is an {\em ergodic} extension if every $T$-invariant
$\X$-measurable function is $\Y$-measurable, i.e. $\I(X,T)\subset
\Y$. Thus Proposition \ref{HK-Pro4.7} implies that
$$\pi^{[d]}_d: (X^{[d]},\mu^{[d]}, T^{[d]})\rightarrow (Z_d^{[d]},\mu_d^{[d]}, T^{[d]}) $$
is $T^{[d]}$-ergodic. That means that $\I^{[d]}(X)=\I^{[d]}(Z_d)$,
and hence $(\Omega_d(X), \I^{[d]}(X), P_d)=(\Omega_d(Z_d),
\I^{[d]}(Z_d), P_d)$. So we can denote the ergodic decomposition of
$\mu^{[d]}_d$ under $T^{[d]}$ by
\begin{equation}\label{b5}
    \mu_d^{[d]}=\int_{\Omega_d} \mu^{[d]}_{d,\omega} \ d P_d(\omega).
\end{equation}
Then by definition
\begin{equation}\label{}
\mu^{[d+1]}_d=\int_{\Omega_d} \mu^{[d]}_{d,\omega}\times
\mu^{[d]}_{d,\omega} \ d P_d(\omega).
\end{equation}
This property is crucial in the proof. Combining (\ref{Omega-d}) and
(\ref{b5}), one has factor maps

\begin{equation}\label{b6}
\begin{split}
    (X^{[d]}, \mu^{[d]}) \stackrel{\pi^{[d]}_d}{\lra}
    (Z_d^{[d]}, \mu^{[d]}_d) \stackrel{\psi}{\lra}(\Omega_d, P_d)
    \end{split}
\end{equation}
Note that $\phi=\psi\circ \pi_d^{[d]}$.

\subsubsection{$\G$-action}\label{G-action}

Now we assume that Theorem A(2) holds for $d\ge 1$. In this
subsection we show the existence of $\G^{[d+1]}$-model.

Let $\pi_d: X\rightarrow Z_d$ be the factor map from $X$ to its
$d$-step nilfactor $Z_d$. By definition, $Z_d$ may be regarded as a
topological system in the natural way. By Weiss's Theorem,
there is a uniquely ergodic model $(\h{X}, \h{\X}, \h{\mu}, T)$ for
$(X,\X, \mu, T)$ and a factor map $\h{\pi_d}: \h{X}\rightarrow Z_d$
which is a model for $\pi_d: X\rightarrow Z_d$.

\[
\begin{CD}
X @>{}>> \h{X}\\
@V{\pi_d}VV      @VV{\h{\pi_d}}V\\
Z_d @>{ }>> Z_d
\end{CD}
\]

\medskip

Hence for simplicity, we may assume that $(\h{X}, \h{\X}, \h{\mu},
T)= (X,\X, \mu, T)$ and $\pi_d=\h{\pi_d}$. Now we show that
$(\Q^{[d+1]}(X), \mu^{[d+1]}, \G^{[d+1]})$ is uniquely ergodic.

Let $\ll$ be a $\G^{[d+1]}$-invariant measure of $\Q^{[d+1]}=\Q^{[d+1]}(X) $. Let
$$p_1: (\Q^{[d+1]}, \G^{[d+1]})\rightarrow (\Q^{[d]}, \G^{[d+1]}); \
{\bf x} =({\bf x'}, {\bf x''})\mapsto {\bf x'}$$
$$p_2: (\Q^{[d+1]}, \G^{[d+1]})\rightarrow (\Q^{[d]},
\G^{[d+1]}); \ {\bf x} =({\bf x'}, {\bf x''})\mapsto {\bf x''}$$ be
the projections. Then $(p_2)_*(\ll)$ is a $\G^{[d+1]}$-invariant
measure of $\Q^{[d]}$. Note that $\G^{[d+1]}$ acts on $\Q^{[d]}$ as
$\G^{[d]}$ actions. By the induction hypothesis,
$(p_2)_*(\ll)=\mu^{[d]}$. Hence let
\begin{equation}\label{a11}
\ll =\int_{\Q^{[d]}} \ll_{{\bf x}}\times \d_{\bf x}\ d \mu^{[d]}
({\bf x})
\end{equation}
be the disintegration of $\ll$ over $\mu^{[d]}$. Since $\ll$ is
$T^{[d+1]}_{d+1}=\id^{[d]} \times T^{[d]}$-invariant, we have
\begin{eqnarray*}
    \ll &= &\id^{[d]} \times T^{[d]} \ll=
    \int_{\Q^{[d]}} \ll_{\bf x}\times T^{[d]}\d_{\bf x} \ d \mu^{[d]}({\bf x})\
    \\ & = & \int_{\Q^{[d]}} \ll_{\bf x}\times \d_{T^{[d]}({\bf x})} \ d
    \mu^{[d]}({\bf x})\\
    \\&=& \int_{\Q^{[d]}} \ll_{(T^{[d]})^{-1}({\bf x})}
    \times \d_{\bf x} \ d \mu^{[d]}({\bf x}).
\end{eqnarray*}
The uniqueness of disintegration implies that
\begin{equation}\label{a12}
    \ll_{(T^{[d]})^{-1}({\bf x})}=\ll_{\bf x},\quad \mu^{[d]} \ a.e.
    \ {\bf x}\in \Q^{[d]}.
\end{equation}

Define $$F: (\Q^{[d]},T^{[d]}) \lra M(X^{[d]}): \ {\bf x}\mapsto
\ll_{\bf x}.$$ By (\ref{a12}), $F$ is a $T^{[d]}$-invariant
$M(X^{[d]})$-value function. Hence $F$ is $\I^{[d]}$-measurable, and
hence $\ll_{\bf x}=\ll_{\phi({\bf x})}, \ \mu^{[d]}\ $ a.e., where
$\phi$ is defined in (\ref{Omega-d}).

Thus by (\ref{a11}) one has that
\begin{equation*}
\begin{split}
    \ll& =\int_{\Q^{[d]}}  \ll_{\bf x}\times \d_{\bf x}\ d\mu^{[d]} ({\bf x})=
\int_{\Q^{[d]}} \ll_{\phi({\bf x})} \times\d_{{\bf x}}\ d\mu^{[d]} ({\bf x})\\
       &=\int_{\Omega_d}\int_{\Q^{[d]}}\ll_{\omega}
       \times \d_{{\bf x}}\ d \mu^{[d]}_\omega({\bf x})
       d P_d(\omega)\\
       &= \int_{\Omega_d}\ll_\omega\times \Big(\int_{\Q^{[d]}} \d_{{\bf x}}\
       d\mu^{[d]}_\omega({\bf x})\Big)d P_d(\omega)\\
       &=\int_{\Omega_d}\ll_\omega\times \mu^{[d]}_\omega \ dP_d(\omega)
\end{split}
\end{equation*}

Let $\pi_d^{[d+1]}: (\Q^{[d+1]}(X), \G^{[d+1]})\lra
(\Q^{[d+1]}(Z_d),\G^{[d+1]})$ be the natural factor map. By Theorem
\ref{HK-Zk}, $(\Q^{[d+1]}(Z_d),\mu_d^{[d+1]})$ is uniquely ergodic.
Hence
\begin{equation*}
\begin{split}
    {\pi_d}^{[d+1]}_*(\ll) =\mu_d^{[d+1]}=\int_{\Omega_d}\mu_{d,\omega}^{[d]}
    \times \mu_{d,\omega}^{[d]} \
    d P_d(\omega).
\end{split}
\end{equation*}
So
\begin{equation}\label{b7}
{\pi_d}^{[d]}_*(\ll_\omega)={\pi_d}^{[d]}_*(\mu^{[d]}_\omega)=\mu_{d,\omega}^{[d]}.
\end{equation}
Note that we have that
$$(p_1)_*(\ll)=(p_2)_*(\ll)=\mu^{[d]},$$
and hence we have
\begin{equation}\label{b8}
    \mu^{[d]}=\int_{\Omega_d}\ll_\omega\ d P_d (\omega)=
\int_{\Omega_d}\mu_\omega^{[d]}\ dP_d(\omega).
\end{equation}
But by (\ref{b7}) and (\ref{b6}) we have
\begin{equation*}
\phi_*(\ll_\omega)=\phi_*(\mu_\omega^{[d]})=\psi_*(\mu_{d,\omega}^{[d]})=\d_\omega.
\end{equation*}
Hence by the uniqueness of disintegration and (\ref{b8}), we have
that $\ll_\omega=\mu_\omega^{[d]}$, $P_d$ a.e. $\omega\in \Omega_d$.
Thus we have
\begin{equation}\label{limit-C2}
\ll_{\Q;d+1}=\int_{\Omega_d}\ll_\omega\times \mu_\omega^{[d]} \ dP_d(\omega)
=\int_{\Omega_d}\mu_\omega^{[d]}\times \mu_\omega^{[d]} \
dP_d(\omega)=\mu^{[d+1]}.
\end{equation} That is, $(\Q^{[d+1]}, \mu^{[d+1]},
\G^{[d+1]})$ is uniquely ergodic. The proof of Theorem A(2)
for $\G$ is completed.

\subsubsection{$\F$-actions}\label{F-action}

Now we assume that Theorem A(1) holds for $d\ge 1$. In this
subsection we show the existence of $\F^{[d+1]}$-model.
We use the same model as in the previous subsection.

Let $\ll$ be a $\F^{[d+1]}$-invariant measure of
$\overline{\F^{[d+1]}}(x^{[d+1]}) $. Let
$$p_1: (\overline{\F^{[d+1]}}(x^{[d+1]}), \F^{[d+1]})
\rightarrow (\overline{\F^{[d]}}(x^{[d]}), \F^{[d+1]}); \ {\bf x}
=({\bf x'}, {\bf x''})\mapsto {\bf x'}$$
$$p_2: (\overline{\F^{[d+1]}}(x^{[d+1]}), \F^{[d+1]})\rightarrow (\Q^{[d]},
\F^{[d+1]}); \ {\bf x} =({\bf x'}, {\bf x''})\mapsto {\bf x''}$$  be
the projections. Note that $$(\overline{\F^{[d]}}(x^{[d]}),
\F^{[d+1]})\simeq (\overline{\F^{[d]}}(x^{[d]}), \F^{[d]})\
\text{and} \ (\Q^{[d]}, \F^{[d+1]})\simeq (\Q^{[d]}, \G^{[d]}).$$
Then $(p_2)_*(\ll)$ is a $\G^{[d]}$-invariant measure of $\Q^{[d]}$.
By subsection \ref{G-action}, $(p_2)_*(\ll)=\mu^{[d]}$. Hence let
\begin{equation}\label{a13}
\ll =\int_{\Q^{[d]}} \ll_{{\bf x}}\times \d_{\bf x}\ d \mu^{[d]}
({\bf x})
\end{equation}
be the disintegration of $\ll$ over $\mu^{[d]}$. Since $\ll$ is
$T^{[d+1]}_{d+1}=\id^{[d]} \times T^{[d]}$-invariant, we have
\begin{eqnarray*}
    \ll &= &\id^{[d]} \times T^{[d]} \ll=
    \int_{\Q^{[d]}} \ll_{\bf x}\times T^{[d]}\d_{\bf x} \ d \mu^{[d]}({\bf x})\
    \\ & = & \int_{\Q^{[d]}} \ll_{\bf x}\times \d_{T^{[d]}({\bf x})} \ d
    \mu^{[d]}({\bf x})\\
    \\&=& \int_{\Q^{[d]}} \ll_{(T^{[d]})^{-1}({\bf x})}
    \times \d_{\bf x} \ d \mu^{[d]}({\bf x}).
\end{eqnarray*}
The uniqueness of disintegration implies that
\begin{equation}\label{a14}
    \ll_{(T^{[d]})^{-1}({\bf x})}=\ll_{\bf x},\quad \mu^{[d]} \ a.e.
\end{equation}

Define $$F: \Q^{[d]} \lra M(\overline{\F^{d}}(x^{[d]})): \ {\bf
x}\mapsto \ll_{\bf x}.$$ By (\ref{a14}), $F$ is a
$T^{[d]}$-invariant $M(\overline{\F^{d}}(x^{[d]}))$-value function.
Hence $F$ is $\I^{[d]}$-measurable, and hence $\ll_{\bf
x}=\ll_{\phi({\bf x})}, \ \mu^{[d]}\ $ a.e. ${\bf x}\in \Q^{[d]}$,
where $\phi$ is defined in (\ref{Omega-d}).

Thus by (\ref{a13}) one has that
\begin{equation*}
\begin{split}
    \ll& =\int_{\Q^{[d]}}  \ll_{\bf x}\times \d_{\bf x}\ d\mu^{[d]} ({\bf x})=
\int_{\Q^{[d]}} \ll_{\phi({\bf x})} \times\d_{{\bf x}}\ d\mu^{[d]} ({\bf x})\\
       &=\int_{\Omega_d}\int_{\Q^{[d]}}\ll_{\omega}
       \times \d_{{\bf x}}\ d \mu^{[d]}_\omega({\bf x})
       d P_d(\omega)\\
       &= \int_{\Omega_d}\ll_\omega\times \Big(\int_{\Q^{[d]}} \d_{{\bf x}}\
       d\mu^{[d]}_\omega({\bf x})\Big)d P_d(\omega)\\
       &=\int_{\Omega_d}\ll_\omega\times \mu^{[d]}_\omega \ dP_d(\omega)
\end{split}
\end{equation*}

Since $(\overline{\F^{d}}(x^{[d]}),\F^{[d]})$ is uniquely ergodic by
assumption, and we let $\nu^{[d]}_x$ be the unique measure. Then
$$(p_1)_*(\ll)=\nu^{[d]}_x, \ \text{and}\ (p_2)_*(\ll)=\mu^{[d]},$$
and hence we have
\begin{equation}\label{f5}
    \nu^{[d]}_x=\int_{\Omega_d}\ll_\omega\ dP_d(\omega).
\end{equation}

Note that we have a factor map $\pi_d^{[d]}:
(\overline{\F^{[d]}}(x^{[d]}),\F^{[d]}, \nu^{[d]}_x)\rightarrow
(\overline{\F^{[d]}}(\hat{x}^{[d]}), \F^{[d]}, \rho_{d,\hat{x}})$,
where $\hat{x}=\pi_d(x)$ and $\rho_{d,\hat{x}}$ as in Theorem \ref{HK-Zk}.
For each $z\in
\overline{\F^{[d]}}(\hat{x}^{[d]})$, let $\eta_z$ be the unique
$T^{[d]}$-invariant measure on $\overline{\O(z,T^{[d]})}$. Then
the map $$ \overline{\F^{[d]}}(\hat{x}^{[d]})\rightarrow
M(\Q^{[d]}(Z_d)); \ z\mapsto \eta_z$$ is a measurable map. This fact
follows from that $z\mapsto \frac{1}{N}\sum_{n<N}\d_{T^nz}$ is
continuous and $\frac{1}{N}\sum_{n<N}\d_{T^nz}$ converges to
$\eta_z$ weakly. Hence we have
\begin{equation}\label{f1}
    \mu^{[d]}_d=\int_{\overline{\F^{[d]}}(\hat{x}^{[d]})}
\eta_z \ d \rho_{d,\hat{x}}(z).
\end{equation}
In fact, it is easy to check that $\displaystyle
\int_{\overline{\F^{[d]}}(\hat{x}^{[d]})} \eta_z \ d
\rho_{d,\hat{x}}(z)$ is $\G^{[d]}$-invariant and hence it is equal
to $\mu^{[d]}_d$ by the uniqueness. Note that (\ref{f1}) is the
``ergodic decomposition'' of $\mu^{[d]}_d$ under $T^{[d]}$, except
that it happens that $\eta_z=\eta_{z'}$ for some $z\neq z'$. Hence
via map $\psi$, we have a factor map
\begin{equation*}
    \Psi: (\overline{\F^{[d]}}(\hat{x}^{[d]}),\rho_{d,\hat{x}})
    \rightarrow (\Omega_d, P_d).
\end{equation*}
And (\ref{f1}) can be rewritten as
\begin{equation}\label{f2}
    \mu^{[d]}_d=\int_{\overline{\F^{[d]}}(\hat{x}^{[d]})}
\eta_z \ d \rho_{d,\hat{x}}(z)=\int_{\Omega_d}\eta_\omega \ d
P_d(\omega)=\int_{\Omega_d} \mu^{[d]}_{d,\omega} \ d P_d(\omega).
\end{equation}

Since we have
\begin{equation}\label{f3}
\begin{split}
    (\overline{\F^{[d]}}(x^{[d]}), \nu^{[d]}_x) \stackrel{\pi^{[d]}_d}{\lra}
    (\overline{\F^{[d]}}(\hat{x}^{[d]}), \rho_{d,\hat{x}}) \stackrel{\Psi}{\lra}(\Omega_d, P_d)
    \end{split}
\end{equation}
we assume that
\begin{equation}\label{f4}
    \nu^{[d]}_x=\int_{\Omega_d}\nu_\omega\ dP_d(\omega)
\end{equation}
is the disintegration of $\nu^{[d]}_x$ over $\Omega_d$.

Let $\pi^{[d+1]}_d: (\overline{\F^{[d+1]}}(x^{[d+1]}),
\F^{[d+1]})\lra (\overline{\F^{[d+1]}}(\hat{x}^{[d+1]}),\F^{[d+1]})$
be the natural factor map. By Theorem \ref{HK-Zk},
$(\overline{\F^{[d+1]}}((\hat{x})^{[d+1]}), \rho_{d+1,\hat{x}})$ is
uniquely ergodic. Let
\begin{equation*}
\begin{split}
    (\pi^{[d+1]}_{d})_*(\ll) =\rho_{d+1,\hat{x}}=\int_{\overline{\F^{[d]}}(\hat{x}^{[d]})}
    \d_z \times \eta_z \
    d\rho_{d,x}(z).
\end{split}
\end{equation*}
be the disintegration of $\rho_{d+1,x}$ over
$\overline{\F^{[d]}}(\hat{x}^{[d]})$. By (\ref{f3}), we have
\begin{equation*}
\begin{split}
     (\pi^{[d+1]}_d)_*(\ll) =\rho_{d+1,\hat{x}}=\int_{\overline{\F^{[d]}}(\hat{x}^{[d]})}
    \d_z \times \eta_z \
    d\rho_{d,x}(z)=\int_{\Omega_d}\rho_{\omega}\times \mu^{[d]}_{d,\omega} \
    dP_d(\omega),
\end{split}
\end{equation*}
where $\rho_{d,\hat{x}}=\int_{\Omega_d}\rho_\omega \ dP_d(\omega)$
is the disintegration of $\rho_{d,\hat{x}}$ over $P_d$. Then
\begin{equation}\label{f6}
    (\pi^{[d]}_d)_*(\ll_\omega)=\rho_\omega, \ \text{and}\ (\pi^{[d]}_d)_*(\mu_\omega^{[d]})
=\mu_{d,\omega}^{[d]}.
\end{equation}

Since $(\pi^{[d]}_d)_*(\nu^{[d]}_x)=\rho_{d,\hat{x}}$, by (\ref{f4})
we have $(\pi^{[d]}_d)_*(\nu_\omega)=\rho_\omega$. Hence by the
uniqueness of disintegration, we have that $\ll_\omega =\nu_\omega$,
$P_d$ a.e.. Thus
\begin{equation}\label{limit-C1}
\ll_{\F;d+1}=\ll=\int_{\Omega_d}\ll_\omega\times \mu_\omega^{[d]} \ dP_d(\omega)
=\int_{\Omega_d}\nu_\omega\times \mu_\omega^{[d]} \ dP_d(\omega).
\end{equation}
That is, $\ll$ is unique and hence
$(\overline{\F^{[d+1]}}(x^{[d+1]}), \F^{[d+1]})$ is uniquely
ergodic. The proof is completed. \hfill $\square$

\section{Proof of Theorem B}\label{sec-AP}
In this section we show Theorem B. We start from the case when $(X,\X, \mu, T)$
is weakly mixing.

\subsection{Preparation}

Let $T: X\rightarrow X$ be a map and $d\in \N$. Set
$$\tau_d=T\times \ldots\times T \ (d \ \text{times}),$$
$$\sigma_d=T\times \ldots \times T^{d}$$ and
$$\sigma_d'=\id \times T\times \ldots \times T^{d-1}=\id\times \sigma_{d-1}.$$
Note that $\langle\tau_d, \sigma_d\rangle=\langle\tau_d,\sigma_d'\rangle$.
For any $x\in \h{X}$, let $N_d(\h{X},x)=\overline{\O((x,\ldots,x),
\langle\tau_d, \sigma_d\rangle)}$, the orbit closure of
$(x,\ldots,x)$ ($d$ times) under the action of the group
$\langle\tau_d, \sigma_d\rangle$. We remark that if $(\h{X},T)$ is
minimal, then all $N_d(\h{X},x)$ coincide, which will be denoted by
$N_d(\h{X})$. It was shown by Glasner \cite{G94} that if $(\h{X},T)$
is minimal, then $(N_d(\h{X}), \langle\tau_d, \sigma_d\rangle)$ is
minimal. Hence if $(N_d(X), \langle\tau_d, \sigma_d\rangle)$ is uniquely ergodic, then it is strictly
ergodic.

\begin{de}\label{de-Furstenberg-selfjoining}
Let $(X,T)$ be a minimal system with $\mu\in M_T(X)$ and $d\ge 1$.
If $(N_d(X), \langle\tau_d,\sigma_d\rangle)$ is uniquely ergodic, then we denote
the unique measure by $\mu^{(d)}$, and call it the {\em Furstenberg
selfjoining}.
\end{de}

Since $(N_d(X), \langle\tau_d,\sigma_d\rangle)$ is uniquely ergodic, it is easy
to see that
\begin{equation*}
\frac {1}{N} \sum_{n=0}^{N-1} \sigma_d^n \mu_\D^d\longrightarrow
\mu^{(d)},\ N\to \infty, \quad \text{weakly in $M(X^d)$},
\end{equation*}
where $\mu_\D^d$ is the diagonal measure on $X^d$ as defined in
\cite{F77}, i.e. it is defined on $X^d$ as follows
\begin{equation*}
    \int_{X^d} f_1(x_1)f_2(x_2)\ldots f_d(x_d)\ d \mu_\D^d(x_1,x_2,\ldots,x_d)=
    \int_X f_1(x)f_2(x)\ldots f_d(x)\ d\mu(x).
\end{equation*}
 Note that if $(N_d(X), \langle\tau_d,\sigma_d\rangle)$ is not
uniquely ergodic, we still can define $\mu^{(d)}$, i.e. generally
one may define $\mu^{(d)}$ as a weak limit point of sequence
$\{\frac {1}{N} \sum_{n=0}^{N-1} \sigma_d^n \mu_\D^d\}$ in $M(X^d)$.
In this case one may have lots of choices for $\mu^{(d)}$.



\subsection{Weakly mixing systems}

In this section we show Theorem B holds for weakly
mixing systems.

\begin{prop}\label{prop-wm-AP}
Let $(X,\X, \mu, T)$ be a weakly mixing dynamical system and $d\in
\N$. If $(X,\X, \mu, T)$ is uniquely ergodic, then $(X^{d}, \X^{d},
\mu^{d}, \langle\tau_d, \sigma_d\rangle)$ is uniquely ergodic, where $\mu^{d}=
\underbrace{\mu\times \ldots \times \mu}_{d\
  \text{times}}$.
\end{prop}

\begin{proof}
We prove the result inductively. It is trivial when $d=1$, since
$\tau_1=\sigma_1=T$.

Now assume the statements hold for $d-1$ ($d\ge 2$), and we show the
case for $d$. Let $\ll$ be a $\langle\tau_d, \sigma_d\rangle$-invariant measure
of $(X^{d}, \X^{d})$. Let $$p_1: X^d=X\times X^{d-1}\rightarrow X; \
{\bf x} =(x_1, {\bf x'})\mapsto x_1$$
$$p_2: X^d=X\times X^{d-1} \rightarrow X^{d-1}; \ {\bf x} =(x_1, {\bf x'})\mapsto {\bf x'}$$  be
the projections. Note that $(p_2)_*(\ll)$ is a
$\langle\tau_{d-1},\sigma_{d-1}\rangle$-invariant measure of $X^{d-1}$. By
inductive assumption, $(p_2)_*(\ll)=\mu^{d-1}$. Let
\begin{equation}\label{b3}
\ll =\int \ll_{{\bf x}}\times \d_{\bf x}\ d \mu^{d-1}({\bf x})
\end{equation}
be the disintegration of $\ll$ over $\mu^{d-1}$. Since $\ll$ is
$\sigma_d'=\id \times \sigma_{d-1}$-invariant, we have
\begin{eqnarray*}
    \ll &= &\sigma'_d \ll=\id \times \sigma_{d-1} \ll=
    \int \ll_{\bf x}\times \sigma_{d-1}\d_{\bf x} \ d \mu^{d-1}({\bf x})\
    \\ & = & \int \ll_{\bf x}\times \d_{\sigma_{d-1}\bf x} \ d \mu^{d-1}({\bf x})\\
    \\&=& \int \ll_{(\sigma_{d-1})^{-1}\bf x}\times \d_{\bf x} \ d \mu^{d-1}({\bf x}).
\end{eqnarray*}
The uniqueness of disintegration implies that
\begin{equation}\label{b4}
    \ll_{(\sigma_{d-1})^{-1}\bf x}=\ll_{\bf x},\quad \mu^{d-1} \ a.e.
\end{equation}

Define $$F: (X^{d-1},\X^{d-1},\sigma_{d-1}) \lra M(X): \ {\bf
x}\mapsto \ll_{\bf x}.$$ By (\ref{b4}), $F$ is a
$\sigma_{d-1}$-invariant $M(X)$-value function. Since $(X,\X,\mu,
T)$ is weakly mixing,
$(X^{d-1},\X^{d-1},\sigma_{d-1})=(X^{d-1},\X^{d-1}, T\times
T^2\times\ldots \times T^{d-1})$ is ergodic and hence $\ll_{\bf
x}=\nu, \ \mu^{d-1}\ $ a.e. for some $\nu\in M(X)$. Thus by
(\ref{b3}) one has that
$$\ll=\int  \ll_{\bf x}\times \d_{\bf x}\ d\mu^{d-1}({\bf x})=
\int \nu \times\d_{\bf x}\ d\mu^{d-1}({\bf x})= \nu\times\mu^{d-1}
.$$ Then we have that $\nu=(p_1)_*(\ll)$ is a $T$-invariant measure
of $X$. By assumption, $\nu=(p_1)_*(\ll)=\mu$. Thus $\ll=\mu\times
\mu^{d-1}=\mu^{d}$. Hence $(X^{d}, \X^{d}, \mu^{d}, \langle\tau_d,
\sigma_d\rangle)$ is uniquely ergodic. The proof is completed.
\end{proof}

\begin{thm}\label{Thm-wm-AP}
If $(X,T)$ is a t.d.s. and $(X,\X, \mu, T)$ is weakly mixing, then it has
a $\langle\tau_d, \sigma_d\rangle-$strictly ergodic model for all $d\in\N$.
\end{thm}

\begin{proof}
By Jewett-Krieger's Theorem , $(X,\X,
\mu, T)$ has a uniquely ergodic model. Without loss of generality,
we may assume that $(X,T)$ itself is a topological minimal system
and $\mu$ is its unique $T$-invariant measure. By Proposition
\ref{prop-wm-AP}, $(X^d, \langle\tau_d,\sigma_d\rangle)$ is uniquely ergodic for
all $d\in \N$. Hence it has a $d$-arithmetic progression strictly
ergodic model.
\end{proof}

\subsection{Nilsystems under action $\langle\tau_d, \sigma_d\rangle$}

Before going on, we need some results on nilsystems under action $\langle\tau_d, \sigma_d\rangle$.

\subsubsection{Basic properties}

In this subsection $d\ge 2$ is an integer, and $(X =
G/\Gamma,\mu_{d-1}, T )$ is an ergodic $(d-1)$-step nilsystem and
the transformation $T$ is translation by the element $t\in G$. Let
$$N_d=N_d(X)=\overline{\O(\Delta_d({X}),
\sigma_d)}=\overline{\O((x,\ldots,x), \langle\tau_d, \sigma_d\rangle)}\subset
X^d$$ and $$N_d[x]=\overline{\O((x,\ldots,x), \sigma'_d)},$$ where
$x\in X$. Then we have

\begin{thm}\cite{BHK05, Z05}\label{ziegler}
With the notations above, we have
\begin{enumerate}
  \item The $(N_d, \langle\tau_d,\sigma_d\rangle)$ is ergodic (and thus uniquely
  ergodic) with some measure $\mu^{(d)}_{d-1}$.
  \item For $\mu$-almost every $x\in X$, the system $(N_d[x],\sigma_d')$ is
uniquely ergodic with some measure $\mu^{(d)}_{d-1,x}$.
  \item $\displaystyle \mu^{(d)}_{d-1} = \int_X \d_x\times \mu^{(d)}_{d-1,x}\
  d\mu(x)$.
  \item (Ziegler) Let $f_1, f_2, \ldots , f_{d-1}$ be continuous functions on
$X$ and let $\{M_i\}$ and $\{N_i\}$ be two sequences of integers
such that $N_i\to \infty$. For $\mu$-almost every $x\in X$,
\begin{equation}
\begin{split}
\frac{1}{N_i} \sum_{n=M_i}^{N_i+M_i-1}
   &f_1(T^nx)f_2(T^{2n}x)\ldots f_{d-1}(T^{(d-1)n}x)\\
   \rightarrow & \int f_1(x_1)f_2(x_2)\ldots f_{d-1}(x_{d-1})\
   d\mu^{(d)}_{d-1,x}(x_1,x_2,\ldots,x_{d-1})
\end{split}
\end{equation}
as $i\to \infty$.
\end{enumerate}
\end{thm}

\subsubsection{The ergodic decomposition of $\mu^{(d)}_{d-1}$ under $\sigma_d$}

Now we study the ergodic decomposition of $\mu^{(d)}_{d-1}$ under
$\sigma_d$. For each $x\in X$, let $\nu^{(d)}_{d-1,x}$ be the unique
$\sigma_d$-invariant measure on $\overline{\O(x^d,\sigma_d)}$, where
$x^d=(x,x,\ldots,x)\in X^d$. Then
$$\varphi: X \longrightarrow
M(N_d);\ \ \ x\mapsto \nu^{(d)}_{d-1,x}$$ is a measurable map. This
fact follows from that $x \mapsto
\frac{1}{N}\sum_{n<N}\d_{\sigma_d^n x^d}$ is continuous and
$\frac{1}{N}\sum_{n<N}\d_{\sigma_d^nx^d}$ converges to
$\nu^{(d)}_{d-1,x}$ weakly. Hence we have
\begin{equation}\label{h9}
    \mu^{(d)}_{d-1}=\int_{X} \nu^{(d)}_{d-1,x} \ d \mu(x).
\end{equation}
In fact, it is easy to check that $\displaystyle \int_{X}
\nu^{(d)}_{d-1,x} \ d \mu(x)$ is $\langle\tau_d,\sigma_d\rangle$-invariant and
hence it is equal to $\mu^{(d)}_{d-1}$ by the uniqueness. Now we
show that (\ref{h9}) is the ``ergodic decomposition'' of
$\mu^{(d)}_{d-1}$ under $\sigma_d$. It is left to show that
$\nu^{(d)}_{d-1,x}\neq \nu^{(d)}_{d-1,y}$ whenever $x\neq y$. This
result will follows from the following fact:
$\overline{\O(x^d,\sigma_d)}\cap
\overline{\O(y^d,\sigma_d)}=\emptyset$ for all $x\neq y$.
In fact, if $\overline{\O(x^d,\sigma_d)}\cap
\overline{\O(y^d,\sigma_d)}\not =\emptyset$, then $y^d\in \overline{\O(x^d,\sigma_d)}$ since
both $\overline{\O(x^d,\sigma_d)}$ and $\overline{\O(y^d,\sigma_d)}$ are minimal.
This means that $(x,y,y,\ldots,y)\in \Q^{[d]}(X)$. Hence $x=y $ by
\cite[Theorem 1.2]{HKM}.

To sum up, we have

\begin{prop}
The algebra $\I(Z_{d-1}^d, \ZZ_{d-1}^d, \mu_{d-1}^{(d)}, \sigma_d)$ of invariant sets under $\sigma_d$ is isomorphic to $\ZZ_{d-1}$.
\end{prop}

\subsection{Proof of Theorem B} Let $(X,T)$ be a strictly ergodic
system and let $\mu$ be its unique $T$-invariant measure.

\subsubsection{Case when $d=1$}\label{AP-d=1}

Now $X^1=X, \tau_1=T, \sigma_1=T$ and $\sigma'_1=\id$. It is trivial
in this case.

\subsubsection{Case when $d=2$}\label{AP-d=2}

In this case $X^2=X\times X$, $\tau_2=T\times T$, $\sigma_2=T\times
T^2$ and $\sigma_2'=\id \times T$. Note that
$\langle\tau_2,\sigma_2\rangle=\G^{[1]}$. Hence it is the same to subsection
\ref{d=1}. In this case $N_2(X)=X\times X$, and its
$\langle\tau_2,\sigma_2\rangle$-uniquely ergodic measure is $\mu\times \mu$.

\subsubsection{Case when $d=3$}

Let $\pi_1: X\rightarrow Z_1$ be the factor map from $X$ to its
Kronecker factor $Z_1$. Since $Z_1$ is a group rotation, it may be
regarded as a topological system in the natural way. By Weiss's Theorem, there is a uniquely ergodic model $(\h{X}, \h{\X},
\h{\mu}, T)$ for $(X,\X, \mu, T)$ and a factor map $\h{\pi_1}:
\h{X}\rightarrow Z_1$ which is a model for $\pi_1: X\rightarrow
Z_1$.

\[
\begin{CD}
X @>{}>> \h{X}\\
@V{\pi_1}VV      @VV{\h{\pi_1}}V\\
Z_1 @>{ }>> Z_1
\end{CD}
\]

\medskip

Hence for simplicity, we may assume that $(\h{X}, \h{\X}, \h{\mu},
T)= (X, \X, \mu, T)$ and $\pi_1=\h{\pi_1}$. Now we show that
$(N_3(X), \langle\tau_3, \sigma_3\rangle)$ is uniquely ergodic.

\bigskip

Before continuing we need some properties about the Kronecker factor
$(Z_1(X), t_1)$ of the ergodic system $(X, \mu, T)$. Recall that
$\mu_1$ is the Haar measure of $Z_1$.

For $s\in Z_1$, let $\xi_{1,s}$ denote the image of the measure
$\mu_1$ under the map $z \mapsto (z, sz^2)$ from $Z_1$ to $Z^2_ 1$.
This measure is invariant under $\sigma_2 = T\times T^2$ and is a
self-joining of the rotation $(Z_1, t_1)$. Let $\xi_s$ denote the
relatively independent joining of $\mu$ over $\xi_{1,s}$. This means
that for bounded measurable functions $f$ and $g$ on $X$,
\begin{equation*}
    \int_{Z_1\times Z_1}
f(x_0)g(x_1) \ d\xi_s (x_0, x_1)=\int_{Z_1} \E(f|\ZZ_1)(z)\E(g |
\ZZ_1)(sz^2) \ d\mu_1(z).
\end{equation*}
where we view the conditional expectations relative to $\ZZ_1$ as
functions defined on $Z_1$.

\medskip

\noindent {\bf Claim:} The invariant $\sigma$-algebra
$\I(\sigma_2)=\I(T\times T^2)$ of $(X \times X, \mu\times \mu, T
\times T)$ is isomorphic to $\ZZ_1$.

\medskip

\noindent {\em Proof of Claim:}
This is a classical result. Here we give a sketch of a proof and later we will give another proof
when we deal with the general case. First by Theorem \ref{thm-F1} we have $K(T^2)=K(T)$, and hence $Z_1$ is the Kronecker factor for both $(X,\X,\mu, T)$ and $(X,\X,\mu, T^2)$.
Let $q_1: (X,\X,\mu, T)\rightarrow (Z_1,\ZZ_1,\mu_1,T)$ and $q_2: (X,\X,\mu, T^2)\rightarrow (Z_1,\ZZ_1,\mu_1,T^2)$ be the factor maps. By Theorem \ref{thm-F3}, if $F\in L^2(X\times X, \mu\times \mu)$ is invariant under $T\times T^2$, then there exists a function $\Phi\in L^2(Z_1\times Z_1, \mu_1\times \mu_1)$ so that $F(x,y)=\Phi(q_1(x), q_2(y))$. That means $\I(X\times X, T\times T^2)$ is measurable with respect to $\ZZ_1\times \ZZ_1$. Hence $\I(X\times X, T\times T^2)=\I(Z_1\times Z_1, T\times T^2)$, which is isometric to $\ZZ_1$. This ends the proof of Claim.

\medskip

Let $\phi: (X\times X, \X\times \X)\rightarrow (\Omega_1, \I^{[1]}, P_1)$ be the
factor map and let $\psi : (\Omega_1, \I^{[1]}, P_1)\rightarrow (Z_1,\ZZ_1,\mu_1)$
be the isomorphic map. Hence we have
\begin{equation}\label{h3}
\begin{split}
    (X\times X, \X\times \X)&\stackrel{\phi}{\lra} (\Omega_1,
    \I^{[1]}, P_1)\stackrel{\psi}{\longleftrightarrow}
    (Z_1,\ZZ_1,\mu_1)\\
    (x,y)&\lra \phi(x,y) \longleftrightarrow s=\psi(\phi(x,y))
\end{split}
\end{equation}

From this, it is not difficult to deduce that the ergodic
decompositions of $\mu_1\times \mu_1$ and $\mu\times \mu$ under
$\sigma_2=T \times T^2$ can be written as
\begin{equation}\label{}
    \mu_1\times \mu_1=\int_{Z_1}\xi_{1,s}\ d \mu_{1}(s);\quad \mu\times \mu=\int_{Z_1}\xi_s\ d \mu_1(s).
\end{equation}
In particular, for $\mu_1$-almost every $s$, the measure $\xi_s$ is
ergodic for $\sigma_2=T \times T^2$.

\bigskip

Now we continue our proof for $d=3$. Let $\ll$ be a
$\langle\tau_3,\sigma_3\rangle$-invariant measure of $N_3(X)$. Let
$$p_1: (N_3(X), \langle\tau_3,\sigma_3\rangle)\rightarrow (X, T); \
(x_1,x_2,x_3)\mapsto x_1$$
$$p_2: (N_3(X), \langle\tau_3,\sigma_3\rangle)\rightarrow (N_2(X), \langle\tau_2, \sigma_2\rangle); \
(x_1,x_2,x_3)\mapsto (x_2,x_3)$$  be the projections. Then
$(p_2)_*(\ll)$ is a $\langle\tau_2,\sigma_2\rangle$-invariant measure of
$N_2(X)=X\times X$. By the case $d=2$, $(p_2)_*(\ll)=\mu\times \mu$.
Hence let
\begin{equation}\label{h1}
\ll =\int_{X^2} \ll_{{(x,y)}}\times \d_{( x,y)}\ d (\mu\times \mu)
(x,y)
\end{equation}
be the disintegration of $\ll$ over $\mu\times \mu$. Since $\ll$ is
$\sigma'_3=\id\times \sigma_2=\id\times T\times T^2$-invariant, we
have
\begin{eqnarray*}
    \ll &= &\id \times \sigma_2 \ll=
    \int_{X^2} \ll_{(x,y)}\times \sigma_2\d_{(x,y)} \ d \mu\times \mu({x,y})\
    \\ & = & \int_{X^2} \ll_{(x,y)}\times \d_{\sigma_2(x,y)} \ d \mu\times \mu(x,y)\\
    \\&=& \int_{X^2} \ll_{(\sigma_2)^{-1}(x,y)}\times \d_{(x,y)} \ d \mu\times\mu({x,y}).
\end{eqnarray*}
The uniqueness of disintegration implies that
\begin{equation}\label{h2}
    \ll_{(\sigma_2)^{-1}(x,y)}=\ll_{(x,y)},\quad \mu\times \mu \ a.e.
\end{equation}

Define $$F: (N_2(X)=X\times X, \sigma_2=T\times T^2) \lra M(X): \
{(x,y)}\mapsto \ll_{(x,y)}.$$ By (\ref{h2}), $F$ is a
$\sigma_2=T\times T^2$-invariant $M(X)$-value function. Hence $F$ is
$\I(\sigma_2)$-measurable, and hence
$\ll_{(x,y)}=\ll_{\phi(x,y)}=\ll_s, \ \mu\times \mu\ $ a.e., where
$\phi$ is defined in (\ref{h3}).

Thus by (\ref{h1}) one has that
\begin{equation*}
\begin{split}
    \ll& =\int_{X^2}  \ll_{(x,y)}\times \d_{(x,y)}\ d\mu\times \mu (x,y)=
\int_{X^2} \ll_{\phi(x,y)} \times\d_{(x,y)}\ d\mu\times \mu (x,y)\\
       &=\int_{Z_1}\int_{X^2}\ll_s\times \d_{(x,y)}\ d\xi_s(x,y)
       d\mu_1(s)\\
       &= \int_{Z_1}\ll_s\times \Big(\int_{X^2} \d_{(x,y)}\
       d\xi_s(x,y)\Big)d\mu_1(s)\\
       &=\int_{Z_1}\ll_s\times \xi_s \ d\mu_1(s)
\end{split}
\end{equation*}

Let $\pi_1^{3}: (N_3(X), \langle\tau_3,\sigma_3\rangle)\lra (N_3(Z_1),
\langle\tau_3,\sigma_3\rangle)$ be the natural factor map. By Theorem
\ref{ziegler}, $(N_3(Z_1), \langle\tau_3,\sigma_3\rangle,\mu^{(3)}_1)$ is
uniquely ergodic. Hence
\begin{equation*}
\begin{split}
    {\pi_1}^{3}_*(\ll) =\mu_1^{(3)}=\int_{Z_1}\d_{s}\times \mu^{(3)}_{1,s} \
    d\mu_1(s).
\end{split}
\end{equation*}
And
$${\pi_1}_*(\ll_s)=\d_s, \ \text{and}\ (\pi_1\times\pi_1)_*(\xi_s)=\mu^{(3)}_{1,s}.$$
Note that we have that
$$(p_1)_*(\ll)=\mu, \ \text{and}\ (p_2)_*(\ll)=\mu\times\mu,$$
and hence we have
$$\mu=\int_{Z_1}\ll_s\ d\mu_1(s).$$
Let $\mu=\int _{Z_1}\rho_s\ d\mu_1(s)$ be the disintegration of $\mu$
over $\mu_1$. Note that ${\pi_1}_*(\ll_s)={\pi_1}_*(\rho_s)=\d_s$, $\mu_1,
 a.e.$. Hence by the uniqueness of disintegration, we have that
$\ll_s=\rho_s$, $\mu_1$ a.e.. Thus
$$\ll=\int_{Z_1}\ll_s\times \xi_s \ d\mu_1(s)=\int_{Z_1}\rho_s\times \xi_s \ d\mu_1(s).$$
That is, $(N_3(X), \langle\tau_3, \sigma_3\rangle)$ is uniquely ergodic.

\subsubsection{Some preparations}

Before going into the proof of the general case, we need some preparations.
Recall the definition of $\mu^{(d)}$ after Definition
\ref{de-Furstenberg-selfjoining}.

\begin{lem}\label{AP-lem-vdc}
Let $(X,\X,\mu,T)$ be an ergodic system and $d\ge 1$ be an integer.
Assume that $f_1,\ldots,f_d\in L^\infty(X,\mu)$ with
$\|f_j\|_\infty\le 1$ for $j=1,\ldots,d$. Then
\begin{equation}\label{AP-VDC}
\lim_{N\to\infty}\Big\|
\frac{1}{N}\sum_{n=0}^{N-1}f_1(T^nx_1)f_2(T^{2n}x_2)\ldots
f_d(T^{dn}x_d) \Big\|_{L^2(\mu^{(d)})}\le \min_{1\le l\le d}\{l\cdot
\interleave f_l\interleave_d \}
\end{equation}
\end{lem}

\begin{proof}
We proceed by induction. For $d=1$, by the Ergodic Theorem,
$$\|\frac{1}{N}\sum_{n=0}^{N-1}T^n f_1 \|_{L^2(\mu)}\to |\int f_1d\mu|=\HK f_1 \HK_1.$$
Let $d\ge 1$ and assume that (\ref{AP-VDC}) holds for $d$. Let
$f_1,\ldots,f_{d+1}\in L^\infty(\mu)$ with $\|f_j\|_\infty\le 1$ for
$j=1,\ldots,d+1$. Choose $l\in \{2,3,\ldots,d+1\}$. (The case $l=1$
is similar). Write
$$\xi_n=\bigotimes_{j=1}^{d+1}T^jf_j=f_1(T^nx_1)f_2(T^{2n}x_2)\ldots
f_{d+1}(T^{(d+1)n}x_{d+1}).$$

By the van der Corput lemma (Lemma \ref{vanderCoput}),
\begin{equation*}
    \limsup_{N\to \infty} \big\| \frac{1}{N}\sum _{n=0}^{N-1}
    \xi_n\big\|^2_{L^2(\mu^{(d+1)})}
    \le \limsup_{H\to \infty}\frac{1}{H} \sum_{h=0}^{H-1}\limsup_{N\to\infty}
    \left|\frac{1}{N}\sum_{n=0}^{N-1} \int \xi_{n+h}\cdot\xi_n d\mu^{(d+1)} \right|.
\end{equation*}
Letting $M$ denote the last $\limsup$, we need to show that $M\le
l^2\HK f_l\HK^2_{d+1}$. For any $h\ge 1$,
\begin{equation*}
\begin{split}
    & \ \ \ \left|\frac{1}{N}\sum_{n=0}^{N-1} \int \xi_{n+h}\cdot\xi_n d\mu^{(d+1)}
    \right| \\ &=\left | \int(f_1\cdot T^hf_1)\otimes \frac{1}{N}\sum_{n=0}^{N-1}
    (\sigma_d)^n \bigotimes_{j=2}^{d+1}f_j\cdot T^{jh}f_jh
    \mu^{(d+1)}(x_1,\ldots,x_{d+1})\right|\\
    &\le \Big\|f_1\cdot T^hf_1 \Big\|_{L^2(\mu^{(d+1)})}\cdot
    \Big\|\frac{1}{N}\sum_{n=0}^{N-1}
    (\sigma_d)^n \bigotimes_{j=2}^{d+1}f_j\cdot T^{jh}f_j
    \Big\|_{L^2(\mu^{(d+1)})}\\
    &= \Big\|f_1\cdot T^hf_1 \Big\|_{L^2(\mu)}\cdot
    \Big\|\frac{1}{N}\sum_{n=0}^{N-1}
    (\sigma_d)^n \bigotimes_{j=2}^{d+1}f_j\cdot T^{jh}f_j
    \Big\|_{L^2(\mu^{(d)})}
\end{split}
\end{equation*}
and by the inductive assumption,
\begin{equation*}
   \left|\frac{1}{N}\sum_{n=0}^{N-1} \int \xi_{n+h}\cdot\xi_n d\mu^{(d+1)}
    \right|\le l\HK f_l\cdot T^{lh}\HK_d.
\end{equation*}
We get
\begin{equation*}
\begin{split}
    M& \le l\cdot \limsup_{H\to\infty}
    \frac{1}{H}\sum_{h=0}^{H-1}\HK f_l\cdot T^{lh}f_l\HK_d
    \le l^2\cdot\limsup_{H\to\infty}\frac{1}{H}\sum_{h=0}^{H-1} \HK f_l\cdot T^h
    f_l\HK_d\\
    &\le l^2\cdot \limsup_{H\to\infty} \Big( \frac{1}{H}\sum_{h=0}^{H-1} \HK f_l\cdot T^h
    f_l\HK_d^{2^d} \Big)^{1/2^d}\\
    &=l^2\cdot \HK f_l\HK_{d+1}^2.
\end{split}
\end{equation*}
The last equation follows from Lemma \ref{lemmaE1}. The proof is completed.
\end{proof}

\begin{lem}\label{AP-lem-vdc2}
Let $(X,\X,\mu,T)$ be an ergodic system and $d\in \N$. Assume that
$f_1,\ldots,f_d\in L^\infty(X,\mu)$. Then
\begin{equation}\label{}
    \E\Big(\bigotimes_{j=1}^d f_j\Big|\I(X^d, \mu^{(d)},\sigma_d)\Big)  =\E\Big( \bigotimes_{j=1}^d
    \E(f_j|\ZZ_{d-1})\Big |\I(X^d, \mu^{(d)},\sigma_d) \Big).
\end{equation}

\end{lem}

\begin{proof}
By Lemma \ref{lem-product}, it suffices to show that
\begin{equation}\label{}
\E\Big(\bigotimes_{j=1}^d f_j\Big|\I(X^d, \mu^{(d)},\sigma_d)\Big)
=0
\end{equation}
whenever $\E(f_k|\ZZ_{d-1})=0$ for some $k\in \{1,2,\ldots,d\}$.
This condition implies that $\HK f_k\HK_d=0$. By the Ergodic Theorem
and Lemma \ref{AP-lem-vdc}, we have
\begin{equation*}
\begin{split}
&\ \ \ \ \Big|\E\Big(\bigotimes_{j=1}^d f_j\Big|\I(X^d,
\mu^{(d)},\sigma_d)\Big)\Big| \\
&=\lim_{N\to\infty}\Big\|
\frac{1}{N}\sum_{n=0}^{N-1}f_1(T^nx_1)f_2(T^{2n}x_2)\ldots
f_d(T^{dn}x_d) \Big\|_{L^2(\mu^{(d)})}\le k \cdot\interleave
f_k\interleave_d =0.
\end{split}
\end{equation*}
So the lemma follows.
\end{proof}

\begin{prop}
Let $(X,\X,\mu,T)$ be ergodic and $d\in \N$. Then the
$\sigma$-algebra $\I(X^d, \mu^{(d)},\sigma_d)$ is measurable with
respect to $\ZZ_{d-1}^{(d)}$.
\end{prop}

\begin{proof}
Every bounded function on $X^{d}$ which is measurable with respect
to \linebreak $\I(X^d, \mu^{(d)},\sigma_d)$ can be approximated in
$L^2(\mu^{(d)})$ by finite sums of functions of the form
$\E(\otimes_{j=1}^d f_j | \I( X^d, \mu^{(d)}, \sigma_d))$ where
$f_1,\ldots,f_d$ are bounded functions on $X$. By Lemma
\ref{AP-lem-vdc2}, one can assume that these functions are
measurable with respect to $Z_{d-1}$. In this case
$\otimes_{j=1}^df_j$ is measurable with respect to
$\ZZ_{d-1}^{(d)}$. Since this $\sigma$-algebra $\ZZ_{d-1}^{(d)}$ is
invariant under $\sigma_d$, $\E(\otimes_{j=1}^df_j|\I(X^d,
\mu^{(d)},\sigma_d))$ is also measurable with respect to
$\ZZ_{d-1}^{(d)}$. Therefore $\I(X^d, \mu^{(d)},\sigma_d)$ is
measurable with respect to $\ZZ_{d-1}^{(d)}$.
\end{proof}

\begin{cor}\label{AP-cor-ergodic}
Let $(X,\X,\mu,T)$ be an ergodic system and $d\in \N$. Then the
factor map $\pi_{d-1}^d: (X^{d}, \mu^{(d)},\sigma_d)\rightarrow
(Z_{d-1}^d, \mu_{d-1}^{(d)},\sigma_d)$ is ergodic.
\end{cor}

\subsubsection{General case}

Now we show the general case.
Assume that Theorem B holds
for $d\ge 1$. We show it also holds for $d+1$.

Let $\pi_{d-1}: X\rightarrow Z_{d-1}$ be the factor map from $X$ to
its $d-1$-step nilfactor $Z_{d-1}$. By definition, $Z_{d-1}$ may be
regarded as a topological system in the natural way. By Weiss's Theorem, there is a uniquely ergodic model $(\h{X}, \h{\X},
\h{\mu}, T)$ for $(X,\X, \mu, T)$ and a factor map $\h{\pi}_{d-1}:
\h{X}\rightarrow Z_{d-1}$ which is a model for $\pi_{d-1}:
X\rightarrow Z_{d-1}$.

\[
\begin{CD}
X @>{}>> \h{X}\\
@V{\pi_{d-1}}VV      @VV{\h{\pi}_{d-1}}V\\
Z_{d-1} @>{ }>> Z_{d-1}
\end{CD}
\]

\medskip

Hence for simplicity, we may assume that $(\h{X}, \h{\X}, \h{\mu},
T)= (X,\X, \mu, T)$ and $\pi_{d-1}=\h{\pi}_{d-1}$. Now we show that
$(X^{d+1},\langle\tau_{d+1},\sigma_{d+1}\rangle)$ is uniquely ergodic. Recall
that $(X^{d},\langle\tau_{d},\sigma_{d}\rangle)$ is uniquely ergodic by the
inductive assumption, and we denote its unique measure by
$\mu^{(d)}$.

\medskip

By Corollary \ref{AP-cor-ergodic}, the factor map $\pi_{d-1}^d:
(X^{d}, \mu^{(d)},\sigma_d)\rightarrow (Z_{d-1}^d,
\mu_{d-1}^{(d)},\sigma_d)$ is ergodic. Hence $\I(X^d,
\mu^{(d)},\sigma_d)=\I(Z_{d-1}^d, \mu_{d-1}^{(d)},\sigma_d)$. By
(\ref{h9}),
$$\displaystyle \mu^{(d)}_{d-1} = \int_{Z_{d-1}} \nu^{(d)}_{d-1,x}\ d \mu_{d-1}(x) $$
is the ergodic decomposition of $\mu^{(d)}_{d-1}$ under $\sigma_d$.
Hence $(X^d, \I(X^d, \mu^{(d)}, \sigma_d))$ is isomorphic to
$(Z_{d-1},\ZZ_{d-1}, \mu_{d-1})$. Let
\begin{equation}\label{h4}
\begin{split}
    (X^d, \X^d,\mu^{(d)})&\stackrel{\phi}{\lra} (X^d,
    \I(X^d, \mu^{(d)},\sigma_d), \mu^{(d)})\stackrel{\psi}{\longleftrightarrow}
    (Z_{d-1},\ZZ_{d-1},\mu_{d-1})\\
    {\bf x}&\lra \phi({\bf x}) \longleftrightarrow s=\psi(\phi({\bf x}))
\end{split}
\end{equation}

From this, we can denote the ergodic decompositions of $\mu^{(d)}$
under $\sigma_d$ by
\begin{equation}\label{}
    \mu^{(d)}=\int_{Z_{d-1}}\nu^{(d)}_s\ d \mu_{d-1}(s).
\end{equation}

\bigskip

Now we continue our proof for $d+1$. Let $\ll$ be a
$\langle\tau_{d+1},\sigma_{d+1}\rangle$-invariant measure of $N_{d+1}(X)$. Let
$$p_1: (N_{d+1}(X), \langle\tau_{d+1},\sigma_{d+1}\rangle)\rightarrow (X, T); \
(x_1,{\bf x})\mapsto x_1$$
$$p_2: (N_{d+1}(X), \langle\tau_{d+1},\sigma_{d+1}\rangle)\rightarrow
(N_d(X), \langle\tau_d, \sigma_d\rangle); \ (x_1,{\bf x})\mapsto {\bf x}$$  be
the projections. Then $(p_2)_*(\ll)$ is a
$\langle\tau_d,\sigma_d\rangle$-invariant measure of $N_d(X)$. By the assumption
on $d$, $(p_2)_*(\ll)=\mu^{(d)}$. Hence let
\begin{equation}\label{h5}
\ll =\int_{X^d} \ll_{{\bf x}}\times \d_{\bf x}\ d \mu^{(d)} ({\bf
x})
\end{equation}
be the disintegration of $\ll$ over $\mu^{(d)}$. Since $\ll$ is
$\sigma'_{d+1}=\id\times \sigma_d$-invariant, we have
\begin{eqnarray*}
    \ll &= &\id \times \sigma_d \ll=
    \int_{X^d} \ll_{\bf x}\times \sigma_d\d_{\bf x} \ d \mu^{(d)}({\bf x})\
    \\ & = & \int_{X^d} \ll_{\bf x}\times \d_{\sigma_d({\bf x})} \ d \mu^{(d)}({\bf x})\\
    \\&=& \int_{X^d} \ll_{(\sigma_d)^{-1}({\bf x})}\times \d_{\bf x} \ d \mu^{(d)}({\bf x}).
\end{eqnarray*}
The uniqueness of disintegration implies that
\begin{equation}\label{h6}
    \ll_{(\sigma_d)^{-1}({\bf x})}=\ll_{\bf x},\quad \mu^{(d)} \ a.e.
\end{equation}

Define $$F: (X^d,\mu^{(d)}, \sigma_d) \lra M(X): \ {(x,y)}\mapsto
\ll_{(x,y)}.$$ By (\ref{h6}), $F$ is a $\sigma_d$-invariant
$M(X)$-value function. Hence $F$ is $\I(X^d,
\mu^{(d)},\sigma_d)$-measurable, and hence $\ll_{\bf
x}=\ll_{\phi({\bf x})}=\ll_s, \ \mu^{(d)}\ $ a.e., where $\phi$ is
defined in (\ref{h4}).

Thus by (\ref{h5}) one has that
\begin{equation*}
\begin{split}
    \ll& =\int_{X^d}  \ll_{{\bf x}}\times \d_{\bf x}\ d\mu^{(d)} ({\bf x})=
\int_{X^d} \ll_{\phi({\bf x})} \times\d_{\bf x}\ d\mu^{(d)} ({\bf x})\\
       &=\int_{Z_{d-1}}\int_{X^d}\ll_s\times \d_{\bf x}\ d\nu^{(d)}_s({\bf x})
       d\mu_{d-1}(s)\\
       &= \int_{Z_{d-1}}\ll_s\times \Big(\int_{X^d} \d_{\bf x}\
       d\nu^{(d)}_s({\bf x})\Big ) d \mu_{d-1}(s)\\
       &=\int_{Z_{d-1}}\ll_s\times \nu^{(d)}_s \ d\mu_{d-1}(s)
\end{split}
\end{equation*}

Let $\pi^{d+1}: (N_{d+1}(X), \langle\tau_{d+1},\sigma_{d+1}\rangle)\lra
(N_{d+1}(Z_{d-1}), \langle\tau_{d+1},\sigma_{d+1}\rangle)$ be the natural factor
map. By Theorem \ref{ziegler}, $(N_{d+1}(Z_{d-1}),
\langle\tau_{d+1},\sigma_{d+1}\rangle,\mu^{(d+1)}_{d-1})$ is uniquely ergodic.
Hence by Theorem \ref{ziegler}
\begin{equation*}
\begin{split}
    \pi^{d+1}_*(\ll) =\mu_{d-1}^{(d+1)}=\int_{Z_{d-1}}\d_{s}\times \mu^{(d+1)}_{d-1,s} \
    d\mu_{d-1}(s).
\end{split}
\end{equation*}
And
$$\pi_*(\ll_s)=\d_s, \ \text{and}\ (\pi^d)_*(\nu^{(d)}_s)=\mu^{(d+1)}_{d-1,s}.$$
Note that we have that
$$(p_1)_*(\ll)=\mu, \ \text{and}\ (p_2)_*(\ll)=\mu^{(d)},$$
and hence we have
$$\mu=\int_{Z_{d-1}}\ll_s\ d\mu_{d-1}(s).$$
Let $\mu=\int _{Z_{d-1}}\theta_s\ d\mu_{d-1}(s)$ be the disintegration
of $\mu$ over $\mu_{d-1}$. Note that
$\pi_*(\ll_s)=\pi_*(\theta_s)=\d_s$, $\mu_{d-1},
 a.e.$. Hence by the uniqueness of disintegration, we have that
$\ll_s=\theta_s$, $\mu_{d-1}$ a.e.. Thus
\begin{equation}\label{limit-D}
\ll_{\tau,\sigma;d+1}=\ll=\int_{Z_{d-1}}\ll_s\times \nu^{(d)}_s \ d\mu_{d-1}(s)=
\int_{Z_{d-1}}\theta_s\times \nu^{(d)}_s \ d\mu_{d-1}(s).
\end{equation} That is,
$(N_{d+1}(X), \langle\tau_{d+1}, \sigma_{d+1}\rangle)$ is uniquely ergodic. The
whole proof is completed.\hfill $\square$

\bigskip

\appendix

\section{Background on Ergodic Theory}

In This Appendix we try to cover notions and results in ergodic
theory which are used in the article. Let $(X,\X,\mu,T)$ be a
measurable system.

\subsubsection{Ergodicity and weak mixing}

First we list some equivalent conditions for ergodicity and weak
mixing.

\begin{thm}
Let $(X,\X,\mu,T)$ be a measurable system. Then the following
conditions are equivalent:
\begin{enumerate}
  \item $T$ is ergodic.
  \item Every measurable function $f$ from $X$ to some Polish Space
  $ P$ satisfying $f\circ T=f\ a.e.$ is of form $f\equiv p$ $a.e.$
  for some point $p\in P$.
  \item $\displaystyle \lim_{N\to\infty} \sum_{n=0}^{N-1}
  \int f\circ T^n\cdot g \ d\mu=\int f\ d\mu
  \int g\ d\mu  $, for all $f,g\in L^2(\mu)$ (or $L^{1}(\mu)$).
\end{enumerate}
\end{thm}

\begin{thm}
Let $(X,\X,\mu,T)$ be a measurable system. Then the following
conditions are equivalent:
\begin{enumerate}
  \item $T$ is weakly mixing.
  \item 1 is the only eigenvalue of $T$ and the geometric multiplicity
  of eigenvalue 1 is 1.
  \item The product system with any ergodic system is still ergodic.
\end{enumerate}
\end{thm}

\subsubsection{Conditional expectation}

If $\Y$ is a $T$-invariant sub-$\sigma$-algebra of $\X$ and $f\in
L^1(\mu)$, we write $\E(f|\Y)$, or $\E_\mu(f|\Y)$ if needed, for the
{\em conditional expectation} of $f$ with respect to $\Y$. The
conditional expectation $\E(f|\Y)$ is characterized as the unique
$\Y$-measurable function in $L^2(Y,\Y, \nu)$ such that
\begin{equation}
    \int_Y g \E(f|\Y)d\nu = \int_X  g\circ \pi f d\mu
\end{equation}
for all $g\in L^2(Y,\Y, \nu)$. We will frequently make use of the
identities $$\int \E(f|\Y) \ d\mu = \int f \ d\mu \quad
\text{and}\quad T \E(f|\Y) = \E(Tf|\Y).$$ We say that a function $f$
is {\em orthogonal} to $\Y$, and we write $f \perp \Y$, when it has a zero
conditional expectation on $\Y$. If a function $f\in L^1(\mu)$ is
measurable with respect to the factor $\Y$, we write $f \in L^1(Y,
\Y, \nu)$.

\medskip

The disintegration of $\mu$ over $\nu$, written as $\mu=\int\mu_y\ d\ \nu(y)$,
is given by a measurable map
$y \mapsto \mu_y$ from $Y$ to the space of probability measures on
$X$  such  that
\begin{equation}
    \E(f|\Y)(y)=\int_X f d\mu_y
\end{equation}
$\nu$-almost everywhere.

\subsubsection{Ergodic decomposition}

Let $x\mapsto \mu_x$ be a regular version of the conditional
measures with respect to the $\sigma$-algebra $\I$. This means that
the map $x\mapsto \mu_x$ is $\I$-measurable, and for very bounded
measurable function $f$ we have $E_\mu (f|\I)(x) = \int f \ d\mu_x$
for $\mu$-almost every $x \in X$. Then the {\em ergodic decomposition} of
$\mu$ is $\mu=\int\mu_x d \mu(x)$. The measures $\mu_x$ have the
additional property that for $\mu$-almost every $x\in X$ the system
$(X,\X,\mu_x , T)$ is ergodic.

\subsubsection{Inverse limit}

We say that $(X,\X, \mu, T)$ is an {\em inverse limit} of a sequence
of factors $(X,\X_j ,\mu, T)$ if $(\X_j)_{j\in\N}$ is an increasing
sequence of $T$-invariant sub-$\sigma$-algebras such that
$\bigvee_{j\in \N}\X_j=\X$ up to sets of measure zero.

\subsubsection{Group rotation}

All locally compact groups are implicitly assumed to be metrizable
and endowed with their Borel $\sigma$-algebras. Every compact group
$G$ is endowed with its Haar measure, denoted by $m_G$.

For a compact abelian group $Z$ and $t\in Z$, we write $(Z, t)$ for
the probability space $(Z,m_Z)$, endowed with the transformation
given by $z \mapsto tz$. A system of this kind is called a {\em
rotation}.

\subsubsection{Joining and conditional product measure}

Let $(X_i, \mu_i,T_i), i = 1,\ldots, k$, be measurable
systems, and let $(Y_i,\nu_i,S_i)$ be corresponding factors, and
$\pi_i:X_i\rightarrow Y_i$ the factor maps. A measure $\nu$ on
$Y=\prod_i Y_i$ defines a {\em joining} of the measures on $Y_i$ if
it is invariant under $S_1\times \ldots \times S_k$ and maps onto
$\nu_j$ under the natural map $\prod_i Y_i\rightarrow Y_j$.

Let $\nu$ be a joining of the measures on $Y_i, i=1,\ldots, k$, and
let $\mu_i=\int \mu_{X_i, y_i}\ d\nu_i(y_i)$ represent the
disintegration of $\mu_i$ with respect to $\nu_i$. Let $\mu$ be a
measure on $X=\prod_i X_i$ defined by
\begin{equation}\label{}
    \mu=\int_Y \mu_{X_1,y_1}\times \mu_{X_2,y_2}\times \ldots \times
    \mu_{X_k,y_k}\ d\nu(y_1,y_2,\ldots,y_k).
\end{equation}
Then $\mu$ is called the {\em conditional product measure with
respect to $\nu$}.

Equivalently, $\mu$ is conditional product measure relative to $\nu$
if and only if for all $k$-tuple $f_i\in L^\infty(X_i,\mu_i),
i=1,\ldots, k$
\begin{equation}
\begin{split}
\int_X
   &f_1(x_1)f_2(x_2)\ldots f_k(x_k)\ d\mu(x_1,x_2,\ldots,x_k)\\
   = & \int_Y \E(f_1|\Y_1)(y_1)\E(f_2|\Y_2)(y_2)\ldots \E(f_{k}|\Y_k)(y_{k})\
   d\nu (y_1,y_2,\ldots,y_{k}).
\end{split}
\end{equation}

\subsubsection{Relatively independent joining}

Let $(X_1,\X_1,\mu_1,T), (X_2,\X_2,\mu_2,T)$ be two systems and let
$(Y,\Y,\nu, S)$ be a common factor with $\pi_i: X_i\rightarrow Y$ for
$i = 1, 2$ the factor maps. Let $\mu_i=\int \mu_{i,y}\ d\nu(y)$
represent the disintegration of $\mu_i$ with respect to $Y$. Let
$\mu_1\times_Y \mu_2$ denote the measure defined by
$$\mu_1\times_Y \mu_2(A)=\int_Y \mu_{1,y}\times \mu_{2,y}\ d\nu(y),$$
for all $A\in \X_1\times \X_2$. The system $(X_1\times X_2,
\X_1\times \X_2,\mu_1\times_Y \mu_2, T\times T)$ is called the
{\em relative product} of $X_1$ and $X_2$ with respect to $Y$ and is
denoted $X_1\times_Y X_2$. $\mu_1\times_Y \mu_2 $ is also
called {\em relatively independent joining} of $X_1$ and $X_2$ over
$Y$.

\subsubsection{Isometric extensions}

Let $\pi: (X,\X,\mu,T)\rightarrow (Y,\Y,\nu,S)$ be a factor map. The
$L^2(X,\X,\mu)$ norm is denoted by $||\cdot||$ and the
$L^2(X,\X,\mu_y)$ norm by $||\cdot||_y$ for $\nu$-almost every $y
\in Y$. Recall $\{\mu_y\}_{y \in Y}$ is the disintegration of $\mu$
relative to $\nu$.

A function $f\in L^2(X,\X,\mu)$ is {\em almost periodic over $\Y$}
if for every $\ep>0$ there exist $g_1,\ldots, g_l \in
L^{2}(X,\X,\mu)$ such that for all $n\in \Z$
$$\min_{1\le j\le l} ||T^nf-g_j||_y<\ep$$
for $\nu$ almost every $y\in Y$. One writes  $f\in AP(\Y)$.
Let $K(X|Y, T)$ be the closed subspace of $L^2(X)$ spanned by the
almost periodic functions over $\Y$. When $\Y$ is trivial,
$K(X,T)=K(X|Y,T)$ is the closed subspace spanned by eigenfunctions
of $T$.

$X$ is an {\em isometric extensions} of $Y$ if $K(X|Y,Y)=L^2(X)$ and
it is a {\em relatively weak mixing extension} of $Y$ if
$K(X|Y,T)=L^2(Y)$.

\begin{thm}\cite[Lemma 6.7.]{F77}\label{thm-F1}
For all $n\in \N$, we have $K(X|Y,T^n)=K(X|Y, T)$.
\end{thm}

\begin{thm}\cite[Theorem 7.1.]{F77}\label{thm-F2}
$\displaystyle K(X_1
\mathop{\times}_{Y}X_2|Y,T)=K(X_1|Y,T)\mathop{\otimes}_{Y}K(X_2|Y,T)$.
\end{thm}

\begin{thm}\cite[Theorem 9.5.]{F77}\label{thm-F3}
Let $k\in \N$. Assume $(X_i,\X_i,\mu_i,T_i)$ is an extension of
$(Y_i,\Y_i,\nu_i,T_i)$ and each $T_i$ has only finitely many ergodic
components for all $i\in \{1,2,\ldots,k\}$. Let $\mu$ is a
conditional product measure with respect to a joining $\nu$ over
$Y_i$. Let $(Y'_i,\Y_i',\nu_i',T_i)$ be the largest isometric
extension of $Y_i$ in $X_i$, and $\pi_i': X_i\rightarrow Y'_i$ be the
factor map for all $i$. Then almost all ergodic components of $\mu$
are conditional product measures relative to $Y'=\prod Y_i'$.

Equivalently, if $F\in L^2(\mu)$ is invariant under $T_1\times
T_2\times\ldots \times T_k$, then there exists a function $\Phi\in
L^2(\prod Y'_i, \nu')$ for $\nu'$ the image of $\mu$, so that
$$F(x_1,x_2,\ldots,x_k)=\Phi(\pi'_1(x_1),\ldots,\pi'_k(x_k)).$$
\end{thm}

\section{The pointwise ergodic theorem for amenable groups}

\subsubsection{}

Amenability has many equivalent formulations; for us, the most
convenient definition is that a locally compact group $G$ is {\em
amenable} if for any compact $K \subset G$ and $\d>0$ there is a
compact set $F\subset G$ such that $$|F\Delta KF|<\d |F|,$$ where we
use both $|\cdot |$ and $m$ to denote the left Haar measure on $G$
(for discrete $G$, we take this to be the counting measure on $G$).
Such a set $F$ will be called {\em $(K, \d)$-invariant}. A sequence
$F_1$, $F_2$, $\ldots$ of compact subsets of $G$ will be called a
{\em F{\rm ${\o}$}lner sequence} if for every compact $K$ and $\d >
0$, for all large enough $n$ we have that $F_n$ is $(K,
\d)$-invariant. Here all groups are assumed to be locally compact
second countable.

\subsubsection{}

Suppose now that $G$ acts bi-measurably from the left by measure
preserving transformations on a Lebesgue space $(X,\mathcal{B},
\mu)$ with $\mu(X) = 1$. We will use for any $f : X \rightarrow \R$
the symbol $\A (F, f )(x)=\A_F(f)$ to denote the average
$$\A(F,f)(x)=\frac{1}{|F|}\int_F f(gx)\ d m(g).$$

\begin{de}
A F{\rm ${\o}$}lner sequence $F_n$ will be said to be {\em tempered}
if for some $C > 0$ and all $n$
\begin{equation}\label{}
    \left|\bigcup_{k\le n}F^{-1}_kF_n\right |\le C\left|F_n\right|.
\end{equation}
\end{de}

\begin{thm}[Lindenstrauss \cite{Lin}]\label{Lindenstrauss} Let $G$ be an amenable
group acting on a measure space $(X,\mathcal{B},\mu)$ by measure
preserving transformation, and let $F_n$ be a tempered F{\rm
{\o}}lner sequence. Then for any $f\in L^1(\mu)$, there is a
$G$-invariant $f^*\in L^1(\mu)$ such that
$$\lim_{n\to \infty} \A(F_n , f)(x) = f^* (x)\quad a.e.$$
In particular, if the $G$ action is ergodic,
$$\lim_{ n\to \infty} \A(F_n , f
)(x) = \int f(x)\ d\mu (x)\quad  a.e.$$

\end{thm}

\section{Uniquely ergodic systems}

In this section we give some conditions for unique ergodicity under
$\Z^d$ actions ($d\in\N$). For completeness a proof is given.

\begin{thm}\label{unique-ergodic}
Let $(X,\Gamma)$ be a topological system, where $\Gamma=\Z^d$. The
following conditions are equivalent.
\begin{enumerate}

  \item $(X,\Gamma)$ is uniquely ergodic.

  \item For every continuous function $f\in C(X)$ the sequence of
  functions
  \begin{equation}\label{}
    \A_Nf(x)=\frac {1}{N^d} \sum_{\gamma \in [0,N-1]^d} f(\gamma x)
  \end{equation}
  converges uniformly to a constant function.

  \item For every continuous function $f\in C(X)$ the sequence of
  functions $\A_Nf(x)$ converges pointwise to a constant function.

 \item There exists a $\mu\in M_\Gamma(X)$ such that
 for all continuous function $f\in C(X)$ and all $x\in X$ the sequence of
  functions \begin{equation}\label{}
    \A_Nf(x)\lra \int f\ d\mu, \ N\to \infty.
  \end{equation}
\end{enumerate}
\end{thm}

\begin{proof}
$``(2)\Rightarrow (3)''$ is obvious.

$``(3)\Rightarrow (4)''$:\quad Define a functional $\Phi:
C(X)\rightarrow \C$ by $$f\mapsto  \lim_{N\to
\infty}\A_Nf(x)=\lim_{N\to \infty}\frac {1}{N^d} \sum_{\gamma \in
[0,N-1]^d} f(\gamma x)$$ Since $\Big |\frac {1}{N^d} \sum_{\gamma
\in [0,N-1]^d} f(\gamma x)\Big|\le \|f\|_\infty$, it is easy to see
that $\Phi$ a continuous linear positive operator. By Riesz
Representation Theorem, there is some $\mu\in M(X)$ such that
$$\Phi(f)=\int f \ d \mu.$$
Since $\Phi(f\circ \gamma)=\Phi(f)$ for all $\gamma\in \Gamma$, we
have $\displaystyle \int f \ d \gamma \mu= \int f\ d \mu$ for all
$f\in C(X)$. Thus $\gamma \mu=\mu$ for all $\gamma \in \Gamma$ and
hence $\mu\in M_\Gamma(X)$.

$``(4)\Rightarrow (1)''$:\quad Let $\nu\in M_\Gamma(X)$. We will
show that $\nu=\mu$. By assumption for all $x\in X$, $\A_Nf(x)=\frac
{1}{N^d} \sum_{\gamma \in [0,N-1]^d} f(\gamma x)\lra \int f\ d\mu, \
N\to \infty.$ By Dominated Convergence Theorem
$$\int f \ d\nu=\lim_{N\to \infty} \int \frac {1}{N^d} \sum_{\gamma \in [0,N-1]^d} f(\gamma x)\ d \nu
=\int\int f \ d\mu d\nu=\int f \ d \mu,$$ for all $f\in C(X)$. Thus
$\nu=\mu$.

$``(1)\Rightarrow (2)''$: \quad If $(2)$ does not hold, then there
is some $g\in C(X)$ and $\ep>0$ such that for any $N\in \N$ there is
some $n>N$ and $x_n\in X$ such that
\begin{equation}\label{a5}
    \Big |  \frac {1}{n^d} \sum_{\gamma \in [0,n-1]^d} g(\gamma x_n)-
    \int g \ d\mu  \Big |\ge \ep.
\end{equation}
Let $\displaystyle \mu_n=\frac{1}{n^d}\sum_{\gamma \in
[0,n-1]^d}\d_{\gamma x_n}=\frac{1}{n^d}\sum_{\gamma \in
[0,n-1]^d}\gamma \d_{x_n}$. Then rewrite (\ref{a5}) as
\begin{equation}\label{a6}
    \Big |  \int g \ d \mu_n -
    \int g \ d\mu  \Big |\ge \ep.
\end{equation}
Take a limit point $\mu_\infty$ of $\{\mu_n\}$ in $M(X)$. Then it is
easy to check that $\mu_\infty\in M_\Gamma(X)$ and by (\ref{a6})
$\mu_\infty\neq \mu$. This contradicts $M_\Gamma(X)=\{\mu\}$. The
proof is completed.
\end{proof}

\section{Nilmanifolds and nilsystems}

\subsubsection{}

Let $G$ be a group. For $g, h\in G$, we write $[g, h] =
ghg^{-1}h^{-1}$ for the commutator of $g$ and $h$ and we write
$[A,B]$ for the subgroup spanned by $\{[a, b] : a \in A, b\in B\}$.
The commutator subgroups $G_j$, $j\ge 1$, are defined inductively by
setting $G_1 = G$ and $G_{j+1} = [G_j ,G]$. Let $k \ge 1$ be an
integer. We say that $G$ is {\em $k$-step nilpotent} if $G_{k+1}$ is
the trivial subgroup.

\subsubsection{}

Let $G$ be a $k$-step nilpotent Lie group and $\Gamma$ a discrete
cocompact subgroup of $G$. The compact manifold $X = G/\Gamma$ is
called a {\em $k$-step nilmanifold}. The group $G$ acts on $X$ by
left translations and we write this action as $(g, x)\mapsto gx$.
The Haar measure $\mu$ of $X$ is the unique probability measure on
$X$ invariant under this action. Let $\tau\in G$ and $T$ be the
transformation $x\mapsto \tau x$ of $X$. Then $(X, T, \mu)$ is
called a {\em $k$-step nilsystem}.

\subsubsection{}

For every integer $j\ge 1$, the subgroup $G_j$ and $\Gamma G_j$ are
closed in $G$. It follows that the group $\Gamma_j=\Gamma \cap G_j$
is cocompact in $G_j$.

\subsubsection{} Here are some basic properties of nilsystems:

\begin{thm}\label{nilsystem}
Let $(X = G/\Gamma,\mu , T )$ be a $k$-step nilsystem with $T$ the
translation by the element $t\in G$. Then:

1. $(X, T )$ is uniquely ergodic if and only if $(X,\mu , T )$ is
ergodic if and only if $(X, T )$ is minimal if and only if $(X, T )$
is transitive.

2. Let $Y$ be the closed orbit of some point $x\in X$. Then $Y$ can
be given the structure of a nilmanifold, $Y = H/\Lambda$, where $H$
is a closed subgroup of $G$ containing $t$ and $\Lambda$ is a closed
cocompact subgroup of $H$.

\medskip

\noindent Assume furthermore that $G$ is spanned by the connected
component of the identity and the element $t$. Then:

\medskip

3. The groups $G_j$, $j\ge 2$, are connected.

4. The nilsystem $(X,\mu, T )$ is ergodic if and only if the
rotation induced by $t$ on the compact abelian group $G/ G_2\Gamma$
is ergodic.

5. If the nilsystem $(X,\mu , T )$ is ergodic then its Kronecker
factor is $Z = G/G_2\Gamma$ with the rotation induced by $t$ and
with the natural factor map $X = G/\Gamma\rightarrow G/G_2\Gamma =
Z$.
\end{thm}

\begin{thm}
Let X =$ G/\Gamma$ be a nilmanifold with Haar measure $\mu$ and let
$t_1,\ldots , t_k$ be commuting elements of $G$. If the group
spanned by the translations $t_1, \ldots , t_k$ acts ergodically on
$(X,\mu)$, then X is uniquely ergodic for this group.
\end{thm}

\section{HK-seminorms}

Let $(X,\mu,T)$ be an ergodic system and $k\in \N$. We write $C:
\C\rightarrow \C$ for the conjugate map $z\mapsto \overline{z}$. Let
$|\ep|=\ep_1+\ldots+\ep_k$ for $\ep\in V_k= \{0,1\}^k$. It is easy
to verify that for all $f\in L^\infty(\mu)$ the integral
$\int_{X^{[k]}}\bigotimes_{\ep\in V_k}
C^{|\ep|}f(x_\ep)d\mu^{[k]}(\bf x)$ is real and nonnegative. Hence
we can define
\begin{equation}\label{}
    \HK f\HK_k=\Big( \int_{X^{[k]}} \bigotimes_{\ep\in V_k}
C^{|\ep|}f(x_\ep)d\mu^{[k]}({\bf x})\Big)^{1/2^k}.
\end{equation}

As $X$ is assumed to be ergodic, the $\sigma$-algebra $\I^{[0]}$ is
trivial and $\mu^{[1]}=\mu \times \mu$. We therefore have
$$\HK f\HK_1=\Big(\int_{X^2}f(x_0)\overline{f(x_1)}d\mu \times\mu(x_0,x_1)\Big)^{1/2}
=\Big|\int fd\mu\Big|.$$

It is showed in \cite{HK05} that $\HK\cdot\HK_k$ is a seminorm on
$L^\infty(\mu)$, and for all $f_\ep\in L^\infty(\mu), \ep\in V_k$,
$$\Big|\int \bigotimes_{\ep\in V_k}f_\ep d\mu^{[k]}\Big|\le \prod_{\ep\in V_k}\HK f_\ep\HK_k.$$
The following lemma follows immediately from the definition of the
measures and the Ergodic Theorem.

\begin{lem}\label{lemmaE1}
For every integer $k\ge 0$ and every $f\in L^\infty(\mu)$, one has
\begin{equation}\label{}
    \HK f\HK_{k+1}=\Big(\lim_{N\to \infty} \frac{1}{N}\sum_{n=0}^{N-1}
    \HK f\cdot T^n \overline{f}\HK_k^{2^k}\Big)^{1/2^{k+1}}.
\end{equation}
\end{lem}

An important property is

\begin{prop}
For a $f\in L^\infty(\mu)$, $\HK f\HK_k=0$ if and only if
$\E(f|\ZZ_{k-1})=0$.
\end{prop}

\section{The van der Corput lemma}

\begin{lem}\label{vanderCoput}
Let $\{x_n\}$ be a bounded sequence in a Hilbert space $\mathcal{H}$
with norm $\parallel \cdot \parallel$ and inner product $<\cdot,
\cdot>$. Then
\begin{equation*}
    \limsup_{N\to \infty} \big\| \frac{1}{N}\sum _{n=1}^N
    x_n\big\|^2
    \le \limsup_{H\to \infty}\frac{1}{H} \sum_{h=1}^H\limsup_{N\to\infty}
    \left|\frac{1}{N}\sum_{n=1}^N <x_n, x_{n+h}>  \right|.
\end{equation*}

\end{lem}

\section{Invariant algebra of $T\times T^2\times \ldots \times T^d$}

\begin{lem}\label{AP-lem-vdc}
Let $(X,\X,\mu,T)$ be an ergodic system, $d\ge 1$ be an integer
and let $\lambda$ be any $d$-fold self-joining of $X$.
Assume that $f_1,\ldots,f_d\in L^\infty(X,\mu)$ with
$\|f_j\|_\infty\le 1$ for $j=1,\ldots,d$. Then
\begin{equation}\label{AP-VDC}
\lim_{N\to\infty}\Big\|
\frac{1}{N}\sum_{n=0}^{N-1}f_1(T^nx_1)f_2(T^{2n}x_2)\ldots
f_d(T^{dn}x_d) \Big\|_{L^2(X^d, \lambda)}\le \min_{1\le l\le d}\{l\cdot
\interleave f_l\interleave_d \}
\end{equation}
\end{lem}

\begin{proof}
We proceed by induction. For $d=1$, the only self-joining $\lambda$ is $\mu$. So
by the Ergodic Theorem,
$$\Big \|\frac{1}{N}\sum_{n=0}^{N-1}T^n f_1 \Big\|_{L^2(\mu)}\to \Big|\int f_1d\mu\Big|=\HK f_1 \HK_1.$$
Let $d\ge 1$ and assume that (\ref{AP-VDC}) holds for $d$ and any $d$-fold self-joining of $X$. Let
$f_1,\ldots,f_{d+1}\in L^\infty(\mu)$ with $\|f_j\|_\infty\le 1$ for
$j=1,\ldots,d+1$. Let $\lambda$ be any $d+1$-fold self-joining of $X$. Choose $l\in \{2,3,\ldots,d+1\}$. (The case $l=1$
is similar). Write
$$\xi_n=\bigotimes_{j=1}^{d+1}T^jf_j=f_1(T^nx_1)f_2(T^{2n}x_2)\ldots
f_{d+1}(T^{(d+1)n}x_{d+1}).$$

By the van der Corput lemma (Lemma \ref{vanderCoput}),
\begin{equation*}
    \limsup_{N\to \infty} \big\| \frac{1}{N}\sum _{n=0}^{N-1}
    \xi_n\big\|^2_{L^2(\ll)}
    \le \limsup_{H\to \infty}\frac{1}{H} \sum_{h=0}^{H-1}\limsup_{N\to\infty}
    \left|\frac{1}{N}\sum_{n=0}^{N-1} \int \xi_{n+h}\cdot\xi_n d\ll \right|.
\end{equation*}
Letting $M$ denote the last $\limsup$, we need to show that $M\le
l^2\HK f_l\HK^2_{d+1}$. For any $h\ge 1$,
\begin{equation*}
\begin{split}
    & \ \ \ \left|\frac{1}{N}\sum_{n=0}^{N-1} \int \xi_{n+h}\cdot\xi_n d\ll
    \right| \\ &=\left | \int(f_1\cdot T^hf_1)\otimes \frac{1}{N}\sum_{n=0}^{N-1}
    (\sigma_d)^n \bigotimes_{j=2}^{d+1}f_j\cdot T^{jh}f_j d
    \ll (x_1,\ldots,x_{d+1})\right|\\
    &\le \Big\|f_1\cdot T^hf_1 \Big\|_{L^2(\ll)}\cdot
    \Big\|\frac{1}{N}\sum_{n=0}^{N-1}
    (\sigma_d)^n \bigotimes_{j=2}^{d+1}f_j\cdot T^{jh}f_j
    \Big\|_{L^2(\ll)}\\
    &= \Big\|f_1\cdot T^hf_1 \Big\|_{L^2(\mu)}\cdot
    \Big\|\frac{1}{N}\sum_{n=0}^{N-1}
    (\sigma_d)^n \bigotimes_{j=2}^{d+1}f_j\cdot T^{jh}f_j
    \Big\|_{L^2(\ll')},
\end{split}
\end{equation*}
where $\ll'$ is the image of $\ll$ to the last $d$ coordinates.
It is clear $\ll'$ is a $d$-fold self-joining of $X$, and by the inductive assumption,
\begin{equation*}
   \left|\frac{1}{N}\sum_{n=0}^{N-1} \int \xi_{n+h}\cdot\xi_n d\ll
    \right|\le l\HK f_l\cdot T^{lh}\HK_d.
\end{equation*}
We get
\begin{equation*}
\begin{split}
    M& \le l\cdot \limsup_{H\to\infty}
    \frac{1}{H}\sum_{h=0}^{H-1}\HK f_l\cdot T^{lh}f_l\HK_d
    \le l^2\cdot\limsup_{H\to\infty}\frac{1}{H}\sum_{h=0}^{H-1} \HK f_l\cdot T^h
    f_l\HK_d\\
    &\le l^2\cdot \limsup_{H\to\infty} \Big( \frac{1}{H}\sum_{h=0}^{H-1} \HK f_l\cdot T^h
    f_l\HK_d^{2^d} \Big)^{1/2^d}\\
    &=l^2\cdot \HK f_l\HK_{d+1}^2.
\end{split}
\end{equation*}
The last equation follows from Lemma \ref{lemmaE1}. The proof is completed.
\end{proof}

\begin{lem}\label{AP-lem-vdc2}
Let $(X,\X,\mu,T)$ be an ergodic system and $d\in \N$. Suppose that $\ll$ is a $d$-fold self-joining of $X$ and it is $\sigma_d$-invariant. Assume that
$f_1,\ldots,f_d\in L^\infty(X,\mu)$. Then
\begin{equation}\label{}
    \E\Big(\bigotimes_{j=1}^d f_j\Big|\I(X^d, \ll,\sigma_d)\Big) =\E\Big( \bigotimes_{j=1}^d
    \E(f_j|\ZZ_{d-1})\Big |\I(X^d, \ll, \sigma_d) \Big).
\end{equation}

\end{lem}

\begin{proof}
By Lemma \ref{lem-product}, it suffices to show that
\begin{equation}\label{}
\E\Big(\bigotimes_{j=1}^d f_j\Big|\I(X^d, \ll, \sigma_d)\Big)
=0
\end{equation}
whenever $\E(f_k|\ZZ_{d-1})=0$ for some $k\in \{1,2,\ldots,d\}$.
This condition implies that $\HK f_k\HK_d=0$. By the Ergodic Theorem
and Lemma \ref{AP-lem-vdc}, we have
\begin{equation*}
\begin{split}
&\ \ \ \ \Big|\E\Big(\bigotimes_{j=1}^d f_j\Big|\I(X^d,
\ll,\sigma_d)\Big)\Big| \\
&=\lim_{N\to\infty}\Big\|
\frac{1}{N}\sum_{n=0}^{N-1}f_1(T^nx_1)f_2(T^{2n}x_2)\ldots
f_d(T^{dn}x_d) \Big\|_{L^2(X^d, \ll)}\le k \cdot\interleave
f_k\interleave_d =0.
\end{split}
\end{equation*}
So the lemma follows.
\end{proof}

\begin{prop}
Let $(X,\X,\nu,T)$ be an ergodic system and $d\in \N$.
Suppose that $\ll$ is a $d$-fold self-joining of $X$ and
it is $\sigma_d$-invariant. Then the
$\sigma$-algebra $\I(X^d, \ll,\sigma_d)$ is measurable with
respect to $\ZZ_{d-1}^{(d)}$.
\end{prop}

\begin{proof}
Every bounded function on $X^{d}$ which is measurable with respect
to $\I(X^d, \ll,\sigma_d)$ can be approximated in
$L^2(X^d, \ll)$ by finite sums of functions of the form \linebreak
$\E(\otimes_{j=1}^d f_j | \I( X^d, \ll, \sigma_d))$ where
$f_1,\ldots,f_d$ are bounded functions on $X$. By Lemma
\ref{AP-lem-vdc2}, one can assume that these functions are
measurable with respect to $Z_{d-1}$. In this case
$\otimes_{j=1}^df_j$ is measurable with respect to
$\ZZ_{d-1}^{(d)}$. Since this $\sigma$-algebra $\ZZ_{d-1}^{(d)}$ is
invariant under $\sigma_d$, $\E(\otimes_{j=1}^df_j|\I(X^d,
\ll,\sigma_d))$ is also measurable with respect to
$\ZZ_{d-1}^{(d)}$. Therefore $\I(X^d, \ll,\sigma_d)$ is
measurable with respect to $\ZZ_{d-1}^{(d)}$.
\end{proof}

\begin{cor}\label{AP-cor-ergodic}
Let $(X, \X, \mu,T)$ be an ergodic system and $d\in \N$. Suppose that $\ll$
is a $d$-fold self-joining of $X$ and it is $\sigma_d$-invariant. Then the
factor map $\pi_{d-1}^d: (X^{d}, \ll, \sigma_d)\rightarrow
(Z_{d-1}^d, \widetilde{\ll},\sigma_d)$ is ergodic, where $\widetilde{\ll}$ is the image of $\ll$.

In particular, one has that $\I(X^{d}, \ll, \sigma_d)$ is isomorphic to $\I(Z_{d-1}^d, \widetilde{\ll},\sigma_d)$.
\end{cor}

\begin{thm}
Let $(X, \X, \mu,T)$ be an ergodic system and $d\in \N$. Suppose that $\ll$
is a $d$-fold self-joining of $X$ and it is $\langle\tau_d,\sigma_d\rangle$-ergodic. Then
$\I(X^{d}, \ll, \sigma_d)$ is isomorphic to $\ZZ_{d-1}$.
\end{thm}

\section{The proof when $d$=2 in Theorem A}\label{sec-H}

\subsubsection{Graph joinings}
Let $\phi: (X,\X,\mu, T)\rightarrow
(Y,\Y, \nu, T)$ be a homomorphism of ergodic systems. Let $\id \times \phi: X\rightarrow X\times Y,
x\mapsto (x,\phi(x))$. Define
\begin{equation}\label{}
    {\rm gr}(\mu,\phi)=\int_X \d_x\times \d_{\phi(x)}\ d\mu(x)=(\id\times
    \phi)_*(\mu).
\end{equation}
It is called a {\em graph joining} of $\phi$. Equivalently, ${\rm
gr}(\mu,\phi)$ is defined by
\begin{equation}\label{}
    {\rm gr}(\mu,\phi)(A\times B)=\mu(A\cap \phi^{-1}B), \ \forall A\in \X, B\in \Y.
\end{equation}

\subsubsection{Kronecker factor $Z_1$} The Kronecker factor of the
ergodic system $(X, \mu, T)$ is an ergodic rotation and we denote it
by $(Z_1(X), t_1)$, or more simply $(Z_1, t_1)$. Let $\mu_1$ denote
the Haar measure of $Z_1$, and $\pi_{X,1}$ or $\pi_1$, denote the
factor map $X \rightarrow Z_1$.

\medskip

For $s\in Z_1$, let $\mu_{1,s}$ denote the image of the measure
$\mu_1$ under the map $z \mapsto (z, sz)$ from $Z_1$ to $Z^2_ 1$,
i.e. $\mu_{1,s}={\rm gr}(\mu_1, s)$. This measure is invariant under
$T^{[1]} = T\times T$ and is a self-joining of the rotation $(Z_1,
t_1)$. Let $\mu_s$ denote the relatively independent joining of
$\mu$ over $\mu_{1,s}$. This means that for bounded measurable
functions $f$ and $g$ on $X$,
\begin{equation}\label{}
    \int_{Z_1\times Z_1}
f(x_0)g(x_1) \ d\mu_s (x_0, x_1)=\int_{Z_1} \E(f|\ZZ_1)(z)\E(g |
\ZZ_1)(sz) \ d\mu_1(z).
\end{equation}
where we view the conditional expectations relative to $\ZZ_1$ as
functions defined on $Z_1$.

It is a classical result that the invariant $\sigma$-algebra
$\I^{[1]}$ of $(X \times X, \mu\times \mu, T \times T)$ consists in
sets of the form
\begin{equation}\label{}
   \{(x,y)\in X\times X: \pi_1(x)-\pi_1(y)\in A\}
\end{equation}
where $A \in \ZZ_1$. Hence $\I^{[1]}$ is isomorphic to $\ZZ_1$. Let
$\phi: (X\times X, \X\times \X)\rightarrow (\Omega_1, \I^{[1]}, P_1)$ be the
factor map and let $\psi : (\Omega_1, \I^{[1]}, P_1)\rightarrow (Z_1,\ZZ_1,\mu_1)$
be the isomorphic map. Hence we have
\begin{equation}\label{Omega1}
\begin{split}
    (X\times X, \X\times \X)&\stackrel{\phi}{\lra} (\Omega_1,
    \I^{[1]}, P_1)\stackrel{\psi}{\longleftrightarrow}
    (Z_1,\ZZ_1,\mu_1)\\
    (x,y)&\lra \phi(x,y) \longleftrightarrow s=\psi(\phi(x,y))
\end{split}
\end{equation}

From this, it is not difficult to deduce that the ergodic
decomposition of $\mu\times \mu$ under $T \times T$ can be written
as
\begin{equation}\label{}
    \mu\times \mu=\int_{Z_1}\mu_s\ d \mu_1(s).
\end{equation}
In particular, for $\mu_1$-almost every $s$, the measure $\mu_s$ is
ergodic for $T \times T$. For an integer $d > 0$ we have
\begin{equation}\label{}
    \mu^{[d+1]}=\int_{Z_1}(\mu_s)^{[d]} \ d\mu_1(s).
\end{equation}
Especially, we have
\begin{equation}\label{}
    \mu^{[2]}=\int_{Z_1}\mu_s\times \mu_s \ d\mu_1(s).
\end{equation}

\subsubsection{$\G^{[2]}$-actions} \label{d=2-G}


Let $\pi_1: X\rightarrow Z_1$ be the factor map from $X$ to its
Kronecker factor $Z_1$. Since $Z_1$ is a group rotation, it may be
regarded as a topological system in the natural way. By Weiss's Theorem, there is a uniquely ergodic model $(\h{X}, \h{\X},
\h{\mu}, T)$ for $(X,\X, \mu, T)$ and a factor map $\h{\pi_1}:
\h{X}\rightarrow Z_1$ which is a model for $\pi_1: X\rightarrow
Z_1$.
\[
\begin{CD}
X @>{}>> \h{X}\\
@V{\pi_1}VV      @VV{\h{\pi_1}}V\\
Z_1 @>{ }>> Z_1
\end{CD}
\]

Hence for simplicity, we may assume that $(\h{X}, \h{\X}, \h{\mu},
T)= (X,\X, \mu, T)$ and $\pi_1=\h{\pi_1}$. Now we show that
$(\Q^{[2]}, \mu^{[2]}, \G^{[2]})$ is uniquely ergodic.

Let $\ll$ be a $\G^{[2]}$-invariant measure of $\Q^{[2]} $. Let
$$p_1: (\Q^{[2]}, \G^{[2]})\rightarrow (\Q^{[1]}, \G^{[2]}); \
{\bf x} =({\bf x'}, {\bf x''})\mapsto {\bf x'}$$
$$p_2: (\Q^{[2]}, \G^{[2]})\rightarrow (\Q^{[1]},
\G^{[2]}); \ {\bf x} =({\bf x'}, {\bf x''})\mapsto {\bf x''}$$  be
the projections. Then $(p_2)_*(\ll)$ is a $\G^{[2]}$-invariant
measure of $\Q^{[1]}=X^{[1]}$. Note that $\G^{[2]}$ acts on
$\Q^{[1]}$ as $\G^{[1]}$ actions. By subsection \ref{d=1},
$(p_2)_*(\ll)=\mu^{[1]}=\mu\times \mu$. Hence let
\begin{equation}\label{a9-0}
\ll =\int_{X^2} \ll_{{(x,y)}}\times \d_{( x,y)}\ d \mu\times
\mu(x,y)
\end{equation}
be the disintegration of $\ll$ over $\mu^{[1]}$. Since $\ll$ is
$T^{[2]}_2=\id^{[1]} \times T^{[1]}$-invariant, we have
\begin{eqnarray*}
    \ll &= &\id^{[1]} \times T^{[1]} \ll=
    \int_{X^2} \ll_{(x,y)}\times T^{[1]}\d_{(x,y)} \ d \mu\times \mu({x,y})\
    \\ & = & \int_{X^2} \ll_{(x,y)}\times \d_{T^{[1]}(x,y)} \ d \mu\times \mu(x,y)\\
    \\&=& \int_{X^2} \ll_{(T^{[1]})^{-1}(x,y)}\times \d_{(x,y)} \ d \mu\times\mu({x,y}).
\end{eqnarray*}
The uniqueness of disintegration implies that
\begin{equation}\label{a10-0}
    \ll_{(T^{[1]})^{-1}(x,y)}=\ll_{(x,y)},\quad \mu^{[1]}=\mu\times \mu \ a.e.
\end{equation}

Define $$F: (\Q^{[1]}=X^{[1]},T^{[1]}) \lra M(X^{[1]}): \
{(x,y)}\mapsto \ll_{(x,y)}.$$ By (\ref{a10-0}), $F$ is a
$T^{[1]}$-invariant $M(X^{[1]})$-value function. Hence $F$ is
$\I^{[1]}$-measurable, and hence $\ll_{(x,y)}=\ll_{\phi(x,y)}=\ll_s,
\ \mu^{[1]}\ $ a.e., where $\phi$ is defined in (\ref{Omega1}).

Thus by (\ref{a9-0}) one has that
\begin{equation*}
\begin{split}
    \ll& =\int_{X^2}  \ll_{(x,y)}\times \d_{(x,y)}\ d\mu\times \mu (x,y)=
\int_{X^2} \ll_{\phi(x,y)}\times\d_{(x,y)}\ d\mu\times \mu (x,y)\\
       &=\int_{Z_1}\int_{X^2}\ll_s\times \d_{(x,y)}\ d\mu_s(x,y)
       d\mu_1(s)\\
       &= \int_{Z_1}\ll_s\times \Big(\int_{X^2} \d_{(x,y)}\
       d\mu_s(x,y)\Big)d\mu_1(s)\\
       &=\int_{Z_1}\ll_s\times \mu_s \ d\mu_1(s)
\end{split}
\end{equation*}

Let $\pi_1^{[2]}: (\Q^{[2]}(X), \G^{[2]})\lra
(\Q^{[2]}(Z_1),\G^{[2]})$ be the natural factor map. By Theorem
\ref{HK-Zk}, $(\Q^{[2]}(Z_1),\mu_1^{[2]})$ is uniquely ergodic.
Hence
\begin{equation*}
\begin{split}
    {\pi_1}^{[2]}_*(\ll) =\mu_1^{[2]}=\int_{Z_1}\mu_{1,s}\times \mu_{1,s} \
    d\mu_1(s).
\end{split}
\end{equation*}
So $$(\pi_1\times\pi_1)_*(\ll_s)=(\pi\times\pi)_*(\mu_s)=\mu_{1,s}.$$
Note that we have that
$$(p_1)_*(\ll)=(p_2)_*(\ll)=\mu^{[1]}=\mu\times\mu,$$
and hence we have
$$\mu\times \mu=\int_{Z_1}\ll_s\ d\mu_1(s)=\int_{Z_1}\mu_s\ d\mu_1(s).$$
Hence by the uniqueness of disintegration, we have that
$\ll_s=\mu_s$, $\mu_1$ a.e.. More precisely, if $\ll_s\not=\mu_s$,
$\mu_1$ a.e., then $\mu_1(\{s\in Z_1: \ll_s\neq \mu_s\})>0$. So
there is some function $f\in C(X\times X)$ such that
\begin{equation*}
    \mu_1\Big(\{s: \ll_s(f)>\mu_s(f)\}\Big)>0.
\end{equation*}
Let $A=\{s: \ll_s(f)>\mu_s(f)\}$. By (\ref{Omega1}), we can consider
$A$ as a subset of $X\times X$:
$$A=\{s: \ll_s(f)>\mu_s(f)\}=\{(x,y)\in X\times X: \ll_{\phi(x,y)}(f)>\mu_{\phi(x,y)}(f)\}.$$
Hence by $\mu\times \mu=\int_{Z_1}\ll_s\ d\mu_1(s)$ we have
\begin{equation*}
\begin{split}
    \mu\times \mu(f\cdot 1_A)&=\int_{X^2}f\cdot 1_A \ d\mu\times
    \mu\\
    &=\int_{Z_1}\int_{X^2}f\cdot 1_A\ d\ll_s(x,y) d\mu_1(s)\\
    &=\int_{Z_1}1_A \int_{X^2}f\ d\ll_s(x,y)\ d\mu_1(s)\\
    &=\int_{A}\ll_s(f)\ d\mu_1(s)
\end{split}
\end{equation*}
Similarly, by $\mu\times \mu=\int_{Z_1}\mu_s\ d\mu_1(s)$ we have
\begin{equation*}
\begin{split}
    \mu\times \mu(f\cdot 1_A)=\int_{A}\mu_s(f)\ d\mu_1(s)
\end{split}
\end{equation*}
Thus
$$0=\int_{A}\ll_s(f)\ d\mu_1(s)-\int_{A}\mu_s(f)\ d\mu_1(s)=
\int_{A}\Big(\ll_s(f)-\mu_s(f)\Big)\ d\mu_1(s)>0,$$ a contradiction!
Hence $\ll_s=\mu_s$, $\mu_1$ a.e., and
$$\ll=\int_{Z_1}\ll_s\times \mu_s \ d\mu_1(s)=\int_{Z_1}\mu_s\times \mu_s \ d\mu_1(s)=\mu^{[2]}.$$
That is, $(\Q^{[2]}, \mu^{[2]}, \G^{[2]})$ is uniquely ergodic. The
proof is completed.

\subsubsection{$\F^{[2]}$-actions}\label{d=2-F}


We use the same model as in the proof of Proposition \ref{d=2-G}.

Let $\ll$ be a $\F^{[2]}$-invariant measure of
$\overline{\F^{[2]}}(x^{[2]}) $. Let
$$p_1: (\overline{\F^{[2]}}(x^{[2]}), \F^{[2]})\rightarrow (\overline{\F^{[1]}}(x^{[1]}), \F^{[2]}); \
{\bf x} =({\bf x'}, {\bf x''})\mapsto {\bf x'}$$
$$p_2: (\overline{\F^{[2]}}(x^{[2]}), \F^{[2]})\rightarrow (\Q^{[1]},
\F^{[2]}); \ {\bf x} =({\bf x'}, {\bf x''})\mapsto {\bf x''}$$  be
the projections. Note that $$(\overline{\F^{[1]}}(x^{[1]}),
\F^{[2]})\simeq (X,T)\ \text{and} \ (\Q^{[1]}, \F^{[2]})\simeq
(X\times X, \G^{[1]}).$$ Then $(p_2)_*(\ll)$ is a
$\G^{[1]}$-invariant measure of $\Q^{[1]}=X^{[1]}$. By subsection
\ref{d=1}, $(p_2)_*(\ll)=\mu^{[1]}=\mu\times \mu$. Hence let
\begin{equation}\label{a9}
\ll =\int_{X^2} \ll_{{(x,y)}}\times \d_{( x,y)}\ d (\mu\times \mu)
(x,y)
\end{equation}
be the disintegration of $\ll$ over $\mu^{[1]}$. Since $\ll$ is
$T^{[2]}_2=\id^{[1]} \times T^{[1]}$-invariant, we have
\begin{eqnarray*}
    \ll &= &\id^{[1]} \times T^{[1]} \ll=
    \int_{X^2} \ll_{(x,y)}\times T^{[1]}\d_{(x,y)} \ d \mu\times \mu({x,y})\
    \\ & = & \int_{X^2} \ll_{(x,y)}\times \d_{T^{[1]}(x,y)} \ d \mu\times \mu(x,y)\\
    \\&=& \int_{X^2} \ll_{(T^{[1]})^{-1}(x,y)}\times \d_{(x,y)} \ d \mu\times\mu({x,y}).
\end{eqnarray*}
The uniqueness of disintegration implies that
\begin{equation}\label{a10}
    \ll_{(T^{[1]})^{-1}(x,y)}=\ll_{(x,y)},\quad \mu^{[1]}=\mu\times \mu \ a.e.
\end{equation}

Define $$F: (\Q^{[1]}=X^{[1]},T^{[1]}) \lra M(X): \ {(x,y)}\mapsto
\ll_{(x,y)}.$$ By (\ref{a10}), $F$ is a $T^{[1]}$-invariant
$M(X)$-value function. Hence $F$ is $\I^{[1]}$-measurable, and hence
$\ll_{(x,y)}=\ll_{\phi(x,y)}=\ll_s, \ \mu^{[1]}\ $ a.e., where
$\phi$ is defined in (\ref{Omega1}).

Thus by (\ref{a9}) one has that
\begin{equation*}
\begin{split}
    \ll& =\int_{X^2}  \ll_{(x,y)}\times \d_{(x,y)}\ d\mu\times \mu (x,y)=
\int_{X^2} \ll_{\phi(x,y)} \times\d_{(x,y)}\ d\mu\times \mu (x,y)\\
       &=\int_{Z_1}\int_{X^2}\ll_s\times \d_{(x,y)}\ d\mu_s(x,y)
       d\mu_1(s)\\
       &= \int_{Z_1}\ll_s\times \Big(\int_{X^2} \d_{(x,y)}\
       d\mu_s(x,y)\Big)d\mu_1(s)\\
       &=\int_{Z_1}\ll_s\times \mu_s \ d\mu_1(s)
\end{split}
\end{equation*}

Let $\pi_1^{[2]}: (\overline{\F^{[2]}}(x^{[2]}), \F^{[2]})\lra
(\overline{\F^{[2]}}((\pi_1(x))^{[2]}),\F^{[2]})$ be the natural
factor map. By Theorem \ref{HK-Zk},
$\overline{\F^{[2]}}((\pi_1(x))^{[2]})$ is uniquely ergodic. Hence
\begin{equation*}
\begin{split}
    {\pi_1}^{[2]}_*(\ll) =\int_{Z_1}\mu_{1}\times \mu_{1,s} \
    d\mu_1(s)=\mu_1^3.
\end{split}
\end{equation*}
And
$${\pi_1}_*(\ll_s)=\mu_1, \ \text{and}\ (\pi_1\times\pi_1)_*(\mu_s)=\mu_{1,s}.$$
Note that we have that
$$(p_1)_*(\ll)=\mu, \ \text{and}\ (p_2)_*(\ll)=\mu^{[1]}=\mu\times\mu,$$
and hence we have
$$\mu=\int_{Z_1}\ll_s\ d\mu_1(s).$$
Let $\mu=\int _{Z_1}\nu_s\ d\mu_1(s)$ be the disintegration of $\mu$
over $\mu_1$. Hence by the uniqueness of disintegration, we have
that $\ll_s=\nu_s$, $\mu_1$ a.e.. Thus
$$\ll=\int_{Z_1}\ll_s\times \mu_s \ d\mu_1(s)=\int_{Z_1}\nu_s\times \mu_s \ d\mu_1(s).$$
That is, $(\overline{\F^{[2]}}(x^{[2]}), \F^{[2]})$ is uniquely
ergodic. The proof is completed.

\medskip


\end{document}